\documentclass[12pt, reqno]{amsart}
\usepackage[utf8]{inputenc}

\usepackage{amssymb}
\usepackage{graphics}
\usepackage{graphicx}
\usepackage{amscd}
\usepackage{color}
\usepackage{amsmath, amsthm, epsfig,  psfrag}
\usepackage{pinlabel}
\usepackage{hyperref}
\usepackage{subfigure}
\usepackage{enumitem}
\usepackage[margin=3cm, marginpar=2.5cm]{geometry}

\newtheorem{theorem}{Theorem}[section]
\newtheorem{lemma}[theorem]{Lemma}
\newtheorem{prop}[theorem]{Proposition}
\newtheorem{cor}[theorem]{Corollary}
\newtheorem{question}[theorem]{Question}

\numberwithin{equation}{section}

\theoremstyle{definition}

\newtheorem{remark}[theorem]{Remark}
\newtheorem{example}[theorem]{Example}

\newcommand{\on}{\operatorname}
\renewcommand{\d}{\partial}

\newcommand{\E}{\mathcal{E}}
\newcommand{\I}{\mathcal{I}}\newcommand{\C}{\mathcal{C}}
\newcommand{\F}{\mathcal{F}}\newcommand{\cS}{\mathcal{S}}
\newcommand{\UU}{\mathcal{U}}
\newcommand{\PP}{\mathcal{P}}
\newcommand{\OO}{\mathcal{O}}

\newcommand{\CC}{\mathbb{C}}
\newcommand{\BB}{\mathcal{B}}
\newcommand{\PPP}{\mathbb{P}}

\newcommand{\Z}{\mathbb{Z}}
\newcommand{\N}{\mathbb{N}}
\newcommand{\QQ}{\mathbb{Q}}

\newcommand{\M}{\operatorname{Mod}}

\newcommand{\cptwobar}{\overline{\CC\textup{P}}\,\!^2}

\renewcommand{\d}{\partial}

\newcommand{\rk}{\on{rk}}

\DeclareMathOperator{\Hom}{Hom}

\title[Stein  fillings vs Milnor fibers]{Stein  fillings vs. Milnor fibers}
\author{R.\,\.{I}nan\c{c} Baykur}
\address{Department of Mathematics and Statistics, University of Massachusetts, Amherst, MA 01003, USA}
\email{inanc.baykur@umass.edu}
\author{Andr\'as N\'emethi}
\address{HUN-REN Alfr\'ed R\'enyi Institute of Mathematics,
Re\'altanoda utca 13-15, H-1053, Budapest, Hungary \newline
  \hspace*{4mm} ELTE - Faculty of Science, Dept. of Geometry,
  P\'azm\'any P\'eter s\'et\'any 1/A, 1117 Budapest, Hungary \newline 
\hspace*{4mm}
   BBU - Babe\c{s}-Bolyai Univ., Str, M. Kog\u{a}lniceanu 1, 400084 
Cluj-Napoca, Romania
    \newline \hspace*{4mm}
BCAM - Basque Center for Applied Math.,
Mazarredo, 14 E48009 Bilbao, Basque Country – Spain}
\email{nemethi@renyi.hu}
\author{Olga Plamenevskaya}
\address{Department of Mathematics, Stony Brook University, Stony Brook, NY 11794, USA}
\email{olga@math.stonybrook.edu}

\begin{document}

\begin{abstract} 
Given a link of a normal surface singularity with its canonical contact structure,  we compare the collection of its Stein fillings to its Milnor fillings (that is, Milnor fibers of possible smoothings). We prove that, unlike Stein fillings, Milnor fillings of a given link have bounded topology; for links of sandwiched singularities, we further establish that there are only finitely many Milnor fillings. We discuss some other obstructions for a Stein filling to be represented by a Milnor fiber, and for various types of singularities, including simple classes like cusps and triangle singularities, we produce Stein fillings that do not come from  Milnor fibers or resolutions.
\end{abstract}

\maketitle

\section{Introduction}

Let $(X,0) \subset \CC^N$ be a complex surface germ with an isolated normal singular point at the origin. The link 
$Y$ of $(X, 0)$ is the smooth $3$-manifold $Y = X \cap \,S_r^{2N-1}$ obtained by intersecting $X$ with a sphere of a small radius $r>0$ centered at 0. The link $Y$ carries a canonical contact structure given by the distribution of complex tangencies to $Y$. The {\em contact link} $(Y, \xi)$ is independent of the choice of $r$, up to contactomorphism. The link $(Y, \xi)$ can also be described as the convex boundary of the negative definite plumbing of symplectic disc bundles over surfaces, given by 
a neighborhood of the exceptional divisor in any resolution of $(X, 0)$. 

In this article, we examine symplectic and Stein fillings of $(Y, \xi)$ and their relation to smoothings of surface singularities with the given boundary link.  Suppose that $\lambda\colon (X, 0) \to (T, 0)$ is a 1-parameter smoothing, so that $X_t:= \lambda^{-1}(t)$ is smooth for $t\neq 0$. It is known that the diffeomorphism type and the Stein homotopy type of $X_t$ is independent of $t$; we refer to $X_t$ as the \emph{Milnor fiber associated to the smoothing $\lambda$}. Intersecting with a small ball centered at the origin, we can also think of a compact version of the Milnor fiber: this is a compact Stein domain $W$ whose contact boundary is given by $(Y, \xi)$.  These, together with the minimal resolution of the singularity, whose symplectic structure can be deformed into a Stein structure by \cite{BO}, can be thought of as Stein fillings of algebro--geometric origin.  We will refer to them as {\em Milnor fillings}, even though  the term is often used for singular surface germs in the literature. It is interesting to compare Milnor fillings 
to the more general Stein fillings of $(Y, \xi)$. These two collections of fillings are known to coincide in certain simple cases, such as lens spaces (arising as links of cyclic quotient singularities) \cite{lisca:lens_fillings, NPP-cycl}, and more generally, links of quotient singularities \cite{PPSU, Bhu-Ono}, and links of simple elliptic  singularities \cite{OhtaOno1, OhtaOno2}. However, for general singularities, they can be different \cite{Akh-Ozbagci1, Akh-Ozbagci2, PS}. 

As a first ground of comparison, we prove that the Milnor fillings of a fixed link have bounded topology:

\begin{theorem} \label{bounds} 
Let $(Y, \xi)$ be the contact link of a normal surface singularity $(X,0)$. Then the Euler characteristic, the signature, $b_2^+$ and $b_2^-$ of all Milnor fillings of $(Y, \xi)$ are bounded by a constant that depends only on the link.
\end{theorem}

\noindent By contrast, this statement is manifestly false for general Stein fillings of links; monodromy techniques as in~\cite{baykur-VHM1, baykur-VHM2, BMVHM, DKP} produce collections of Stein fillings with unbounded topology for many non-rational singularities---we will come back to this point shortly.

It is important to note that in Theorem~\ref{bounds}, we fix the link $(Y, \xi)$, which amounts to fixing only the {\em topological} type of the surface singularity. A given topological type may support infinitely many {\em analytic types} $(X, 0)$. In fact, this is almost always the case: topological types that support a unique (resp. finitely many) analytic structures satisfy very restrictive constraints on their resolution graphs~\cite{Lauf}. These are called {\em taut} (resp. {\em pseudotaut}); most singularities are not (pseudo)taut.
For a fixed analytic type $(X, 0)$ there are only finitely many Milnor fibers, but different analytic types may produce non-diffeomorphic Milnor fibers. By~\cite{CNPP}, all these Milnor fibers induce the unique Milnor-fillable contact structure $\xi$ on the link, which in fact can be identified as the boundary of the plumbing of symplectic disk bundles given by the resolution \cite{GaySt}.  The resulting collection of Milnor fillings may a priori be infinite.  We do not know if an infinitude of Milnor fillings can actually exist, but for \emph{sandwiched singularities}, an important subclass of rational singularities, we are able to prove the finiteness of Milnor fillings, sharpening Theorem~\ref{bounds} for this class. Note that most  sandwiched singularities are not (pseudo)taut. 

\begin{theorem} \label{thm:sandwich-finite} 
The contact link of a sandwiched surface singularity admits only finitely many Milnor fillings, up to diffeomorphism.
\end{theorem}

The proof is based on a dimensional reduction due to de Jong--van Straten \cite{dJvS}. A sandwiched singularity can be characterized via an associated germ of a singular plane algebraic curve, and the Milnor fibers can be described via appropriate deformations of the curve germ.  
For curves, the parameter space of the analytic structures in a given topological type is finite-dimensional and much better understood. These enable us to prove Theorem~\ref{thm:sandwich-finite} by working with the corresponding parameter space for all analytic structures in the given topological type of the singularity. 

Moving on to Stein fillings of links, we probe the following problem:

\begin{question} \cite{Nem} \label{Milnor=Stein}
For which contact links of normal surface singularities is every Stein filling diffeomorphic to a Milnor filling or the minimal resolution? 
\end{question}

As in~\cite{PS}, we are going to call a Stein filling of a link  \emph{unexpected} if it is not diffeomorphic to any Milnor filling or the minimal resolution---though a takeaway from our work here will be that one should generally ``expect the unexpected'' in this context. 

The examples of unexpected fillings we build in this article will come from monodromy factorizations through a variety of manipulations. Recall that a factorization of the monodromy of the supporting open book on a contact $3$--manifold into a product of positive Dehn twists gives rise to a Lefschetz fibration. Provided all the Dehn twist curves in the factorization are homologically essential curves on the fiber, the \emph{allowable Lefschetz fibration} carries a Stein structure and provides a filling of the contact manifold. 

Theorem~\ref{bounds} implies that whenever the link admits a family of Stein fillings with unbounded topology, all but finitely many of these fillings will be unexpected. Examples of monodromies with arbitrarily long factorizations given in~\cite{baykur-VHM2, BMVHM, DKP}
produce Stein fillings with arbitrarily high Euler characteristic. This approach works for a large family of singularities with dual resolution graphs as in the next theorem. Every normal surface singularity has a {\em good} resolution where all exceptional curves are smooth, and all intersections between exceptional curves occur in transverse double points only.

\begin{theorem}~\label{thm:examples} 
Suppose that $(X, 0)$ is a  normal surface singularity with a good resolution such that:
\begin{enumerate}[label=(\roman*)]
\item for every irreducible exceptional curve $E$, its self-intersection $E \cdot E$ and the number $a(E)$ of intersections with
the other exceptional curves satisfy the inequality 
\begin{equation} \label{goodvert}
a(E) \leq - E \cdot E
\end{equation}
\item there is an exceptional curve $E_0$ of genus $g > 2$ such that $- E_0 \cdot E_0 \leq 2g -4$.
\end{enumerate}
Then the link  $(Y, \xi)$ of $(X, 0)$ has infinitely many Stein fillings that are 
not homotopy equivalent to any Milnor fibers of any normal surface singularities with the given link. For an infinite family of graphs, there are such Stein fillings with $b_1=0$.
\end{theorem}

A deep result of Greuel--Steenbrink~\cite{GreSt} tells that any Milnor fiber $W$ of a normal surface singularity has vanishing first Betti number, which gives a way to argue that certain Stein fillings cannot come from Milnor fibers~\cite{Akh-Ozbagci1}. Many of our unexpected fillings pass the Greuel--Steenbrink test. 

The proof of Theorem~\ref{bounds} is based on the classical formulas of Laufer, Durfee and Wahl \cite{Durf,Laufermu,Steenbrink,Wahl-smooth} expressing the Euler characteristic and signature of Milnor fibers in terms of the geometric genus of the singularity. The geometric genus depends on the analytic structure, and thus in general {\em not} determined by the link, but it satisfies {\em topological} upper bounds that yield the result for all Milnor fillings. These upper bounds, and therefore those of Theorem~\ref{bounds}, can be computed from the resolution graph in specific examples.  We can get explicit unexpected fillings in Theorem~\ref{thm:examples} by producing monodromy factorizations that violate these bounds. A simple family is the hypersurface singularity $\{x^p+y^p+z^p=0\} \subset \mathbb{C}^3$. When $p\geq 5$, its minimal resolution has an exceptional divisor consisting of a single curve $E$ of genus $g=\frac12(p-1)(p-2)>2$, with self-intersection $-b=-p$. The link is the total space of an oriented $S^1$--bundle over $\Sigma_g$ with Euler number $-b$, equipped with the Boothby--Wang contact structure, which is supported by an open book whose page has genus $g$ and $b$ boundary components, and the monodromy is a boundary multitwist. This monodromy admits arbitrarily long factorizations into Dehn twists along homologically essential curves~\cite{BMVHM}, generating an infinite family of unexpected fillings with arbitrarily large topology. 

The singularities in Theorem~\ref{thm:examples} are all non-rational, due to the presence of the higher-genus exceptional curves. For links of \emph{rational singularities}, which are defined as singularities of geometric genus $p_g=0$, the topology of general Stein fillings is  quite restricted. From the topological perspective, these links are characterized as \emph{L-spaces} in the sense of Heegaard Floer or monopole Floer homology \cite{Nem-Lspace}. In this case, all Stein, and more generally, all minimal symplectic fillings are known to be negative definite, with uniformly bounded topology \cite{OSgenus, Stip1}. Nonetheless, unexpected fillings often 
exist, even in the most restricted case of rational singularities with \emph{reduced fundamental cycle}, which can be characterized by having contact links supported by planar open books \cite{NT, GGP}. For links of such singularities, unexpected fillings were constructed in~\cite{PS}.  (Note that in that paper, the unexpected fillings were distinguished from the Milnor fibers only relative to certain boundary data, but one can detect them in the absolute sense by using standard facts on the mapping class group of the link, \cite{PS2}.)  Some of these unexpected fillings and Milnor fibers are simply connected and have negative definite intersection forms of the same rank. 


The next interesting class is 
{\em minimally elliptic singularities}, which are Gorenstein singularities of geometric genus $p_g=1$. If we additionally assume that the link is a rational homology sphere, 
 the links of such singularities are \emph{almost L-spaces} in the terminology of Heegaard Floer or monopole Floer homology, and by the work of Lin~\cite{Lin}, the  \emph{indefinite} Stein fillings of each link all have the same Euler characteristic and signature. Lin's results do not cover the negative-definite case, and it is intriguing that we can find such unexpected fillings distinguished by the rank of their intersection forms, using further constraints on the topology of Milnor fibers due to Laufer.

\begin{theorem} \label{thm:unexpected-min-elliptic}
There are infinite families of minimally elliptic singularities whose links admit Stein fillings that cannot be homotopy equivalent to a Milnor fiber. The examples include non-smoothable cusp singularities and triangle singularities. 
\end{theorem}

Minimally elliptic singularities with reduced fundamental cycle (in their minimal resolution) are known to fall into three special subclasses of singularities: \emph{simple elliptic}, \emph{cusp}, and \emph{triangle singularities}. In the case of simple elliptic singularities, which by definition have a resolution with a single exceptional curve of genus $1$, every minimal symplectic filling of the contact link is diffeomorphic to a Milnor fiber or the minimal resolution by the work of Ohta--Ono \cite{OhtaOno2}. Thus, they constitute a class with a positive answer to Question~\ref{Milnor=Stein}. In contrast, we establish that there are many links of \emph{non-smoothable} cusp singularities (which therefore admit no Milnor fibers) that admit Stein fillings not diffeomorphic to the minimal resolution, as well as many links of (smoothable) triangle singularities that also admit unexpected Stein fillings.

Upon closer examination, however, one notices that at least some of our examples in Theorem~\ref{thm:unexpected-min-elliptic} could still be expected from the algebro-geometric viewpoint. Indeed, as symplectic fillings, these manifolds are obtained by symplectic rational blowdown,  and  it is plausible that on the algebraic side they can be generated via  {\em P-resolutions} \cite{Kol2}. (This was pointed out to us by Paul Hacking.) It would be interesting to know if Theorem~\ref{bounds} extends to this case, that is, if the collection of fillings generated by P-resolutions has bounded topology.

 In all the examples of unexpected Stein fillings discussed above for rational and minimally elliptic singularities, where we know that the Stein fillings of the links necessarily have uniformly bounded topology,
we only had a finite collection of examples for each given link. This brings up the more general question of whether any contact $3$--manifold with uniform bounds on the topology of its Stein fillings can have only finitely many Stein fillings, up to diffeomorphism \cite{baykur-VHM1}. Infinite collections of fillings for links of rational singularities 
seem particularly hard or perhaps impossible to generate:





\begin{question} \label{question:finite} Is there a link of rational singularity that admits infinitely many, pairwise non-diffeomorphic Stein fillings? Minimal symplectic fillings? 
\end{question}

We may more generally ask if there are any L-spaces that admit infinitely many Stein fillings; no examples are known.  Finiteness results are only available for some special cases, such as lens spaces, where fillings have been fully classified \cite{lisca:lens_fillings}. Otherwise, the question is open even for rational singularities with reduced fundamental cycle. 

Instead of Stein fillings, one may more generally consider \emph{minimal symplectic fillings}. For some classes of contact $3$--manifolds, as as the planar ones, the collections of Stein fillings and minimal symplectic fillings coincide (we consider fillings up to a diffeomorphism). In particular, all minimal symplectic fillings of lens spaces are given by Milnor fibers. However, in general, minimal symplectic fillings are  much easier to generate:  even in the minimally elliptic case when Stein fillings have bounded topology, we can easily construct symplectic fillings with arbitrarily large Euler characteristics. These can be obtained by taking a symplectic fiber sum  \cite{gompf:fibersum} along an embedded symplectic torus of square zero in a Stein filling $W$ of $(Y, \xi)$, where minimality can be argued using \cite{usher:minimality} and \cite{lisca-matic}. Moreover, one can perform knot surgery \cite{fintushel-stern:knotsurgery} with fibered knots along such torus to produce exotic copies. 
It is worth noting that this construction does not work for rational singularities; the presence of a square zero symplectic torus in a Stein filling $W$ implies that $W$ has an indefinite intersection form, which cannot be the case when the singularity is rational.  This fact,  together with the relation to taut foliations discussed below, underscores the challenge in getting potential examples for Question~\ref{question:finite}.

The existence of infinite collections of minimal symplectic fillings is related to taut foliations and the L-space conjecture \cite{BGW}, as reminded to us by Jeremy Van Horn-Morris. A link of a normal surface singularity is an L-space if and only if the singularity is rational \cite{Nem-Lspace}. The L-space conjecture was proved in~\cite{HRRW, SaRas} for all graph manifolds; it follows that link $Y$ carries always a taut foliation in the non-rational case. One can then get a symplectic structure on $Y \times [-1, 1]$ and 
cap off $Y \times \{-1\}$ by a {\em concave} filling \`a la Eliashberg \cite{Eliash-capoff}, which results in a (convex) symplectic filling for $Y \times \{1\}$. Moreover, one can use different concave fillings (with arbitrarily large Euler characteristics) to get infinitely many different minimal symplectic fillings. (In general, one needs to check that the canonical contact structure on the link is obtained as a perturbation of the taut foliation. This would be automatic in examples where the tight contact structure is unique, e.g. for  $Y=\Sigma(2, 3, 6m+1)$, see \cite{MaTo}. For Seifert fibered spaces over $S^2$, one can use a transverse taut foliation to get a transverse contact structure $\xi$, see \cite{LiStip07}, so  $\xi$ must be canonical by \cite{CNPP}.) 

It is often conjectured that planar contact 3-manifolds admit only finitely many minimal symplectic fillings, up to symplectic deformation \cite{LiWe}. One might conjecture that finiteness of fillings generally holds for all rational singularities, or even for all L-spaces.

\smallskip
 \noindent 
\textit{Organization of the paper:} We discuss the constraints on the topological invariants of Milnor fibers in Section~\ref{sec:bounds} and prove Theorem~\ref{bounds}. We prove Theorem~\ref{thm:examples} in Section~\ref{s:examples} and present, for certain families of links, explicit examples of non-Milnor Stein fillings with arbitrarily large topology. Section~\ref{s:min-elliptic} is where we consider the minimally elliptic singularities and construct many examples of unexpected Stein fillings to prove Theorem~\ref{thm:unexpected-min-elliptic}. 
In Section~\ref{s:sa}, we show that there are finitely many Milnor fillings of a sandwich singularity, proving Theorem~\ref{thm:sandwich-finite}. 

 
\enlargethispage{0.1in}
\smallskip
\noindent \textit{Acknowledgements.} We thank Paul Hacking and Jeremy Van Horn-Morris for some interesting discussions, and Jie Min for his comments on the first version of this article. We are grateful to  Kouichi Yasui for pointing out a gap in the first version.
RIB was supported by the NSF grant DMS-2005327. AN was supported by NKFIH Grant \emph{``\'Elvonal (Frontier)''} KKP 144148. OP was supported by NSF grants DMS-2304080 and DMS-1906260. 
The authors would like to thank the organizers of the Nantes workshop on \emph{``Complex and symplectic curve configurations''} in December 2022, where this collaboration was initiated. We are also grateful for the generous  support and the excellent research environment provided by the Erd\H{o}s Center during the Budapest special semester on \emph{``Singularities and low dimensional topology''} in Spring 2023.

\smallskip
\section{Proof of Theorem~\ref{bounds}} \label{sec:bounds}

We first recall some useful inequalities   that follow   from the exact sequence of the pair $(W, \partial W)$, where $W$ is a Stein filling of the link $Y= \partial W$. Since $W$ is Stein and so $H_3(W)=0$ and $H_1(W, \partial W)= H^3(W)=0,$ we have
$$
0 \rightarrow H_3(W, \d W) \stackrel{\d_*}{\longrightarrow} H_2 (\d W) \stackrel{i_*}{\longrightarrow} H_2(W) \longrightarrow  H_2 (W, \d W) \longrightarrow H_1 (\d W) \longrightarrow H_1 (W) \rightarrow 0.
$$

Then,
\begin{equation}\label{bettiYW}
b_1(W)\leq b_1(\d W) = b_1(Y).
\end{equation}
Furthermore, the intersection form $Q_W$ is zero on $i_*(H_2(\d W))$, and, moreover,
by the Poincar\'e--Lefschetz duality $i_*(H_2(\d W))$ is the maximal subspace of $H_2(W)$ where $Q_W$ is zero. Therefore,
\begin{equation}
b_2^0(W)= \rk i_*(H_2(\d W))=
\rk  H_2(\d W) - rk H_3 (W, \d W) = b_1(\d W) - b_1(W);
\label{bees}
\end{equation}
we used the fact that $\d_*$ is injective.  It follows that
\begin{equation} \label{b20}
 b_2^0(W) \leq b_1(Y)
\end{equation}
for every Stein filling of the link $Y$. When $W$ is a Milnor fiber, $b_1(W)=0$ by~\cite{GreSt},
so in that case
\begin{equation}\label{b20=b1}
b_2^0(W)=b_1(Y).
\end{equation}

We are going to use the following result of Stipsicz~\cite{Stip1} providing uniform bounds on the topology of arbitrary Stein fillings:

\begin{prop}\cite[Theorem 1.4]{Stip1} 
\label{stip-bound} 
For a given contact 3-manifold $(Y, \xi)$, the quantity $2\chi(W)+3 \sigma(W)$ is bounded below as $W$ ranges over all Stein fillings of $(Y, \xi)$.
\end{prop}

\begin{remark}
If $(Y, \xi)$ is given by an open book decomposition, and one considers a family of Stein fillings
given by monodromy factorizations, it is possible to give explicit bounds in the above statement.
For example, let $Y_{g,b}$ be an oriented $S^1$--bundle over $\Sigma_g$ with Euler number $-b$
with its Boothby-Wang contact structure, as in the end of the previous section. It is supported by the
open book with page $\Sigma_g^b$ and monodromy the boundary multitwist $\Delta$ in the mapping class group $\M(\Sigma_g^b)$. Capping off every boundary component of this open book a l\`{a} Eliashberg \cite{Eliash-capoff}, we obtain a symplectic cobordism to a $\Sigma_g$--bundle over $S^1$ with trivial monodromy. The latter can be symplectically capped of with a symplectic trivial $\Sigma_g$--bundle over $\Sigma_1^1$. This way, we generated a concave cap
$\mathcal{C}$ for $Y_{g,b}$. Now, let $X$ be a genus $g \geq 2$ allowable Lefschetz fibration over $D^2$ which has the boundary open book with monodromy $\Delta=1$ in $\M(\Sigma_g^b)$ with vanishing cycles on $\Sigma_g^b$ that map to homologically essential curves under an embedding $\Sigma_g^b \hookrightarrow \Sigma_g$. Then,
the cap $\mathcal{C}$ extends this fibration to a relatively minimal symplectic Lefschetz fibration on $Z:=X \cup \mathcal{C}$ over $T^2$. Since $g \geq 2$, we already have $b_2^+(\mathcal{C})>1$, so $b_2^+(Z)>1$. And since a relatively minimal Lefschetz fibration over a positive genus surface is always minimal \cite{stipsicz:minimal},
it now follows from Taubes~\cite{Taubes} that $c_1^2(Z) \geq 0$. Here we have
\[
c_1^2(\mathcal{C})= 2 \chi(\mathcal{C}) +3 \sigma(\mathcal{C})=2(b+2g-2)+0=2(2g+b-2) \, .
\]
It follows from $c_1^2(Z)=c_1^2(X)+c_1^2(\mathcal{C})$ that
\begin{equation} \label{c12}
c_1^2(X) = 2 \chi(X) + 3 \sigma(X) \geq -2(2g+b-2)
\end{equation}
for these special Stein fillings. A similar argument can be used more generally, to find explicit bounds for $2\chi(W) + 3 \sigma(W)$ whenever one starts an explicit open book monodromy for $Y=\partial W$.   
\end{remark}

Now, we turn to the situation where a Stein filling is given by a Milnor fiber. In this case,
there are further constraints on topology: topological invariants of Milnor fibers of smoothings of a normal surface singularity $(X, 0)$ have a subtle relation to analytic invariants of its resolution, namely to geometric genus. We assume that  $\tilde{X} \to X$ is a good resolution.
The {\em geometric genus} of a normal surface singularity is defined as
$$
p_g= h^1(\tilde{X}, \mathcal{O}_{\tilde{X}});
$$
it is known that $p_g$ is independent of the choice of
the resolution, see \cite{Nem-book}. Let
$\mu= \text{rank}\, H_2 (W)=b_2(W)$ denote the \emph{Milnor number} of $W$, write $\mu(W)= b_2(W)= b_2^+(W) + b_2^-(W) + b_2^0(W)$, and let  $\sigma(W)= b_2^+(W) - b_2^-(W)$ denote the signature of $W$.

\begin{prop}[\cite{Durf,Laufermu,Steenbrink}, see also  \cite{Wahl-smooth}] \label{durfee} 
The Milnor fiber $W$ of a smoothing of a
normal surface singularity $(X, 0)$  with the link $Y$ satisfies
\begin{align}
\label{durf1}  2 p_g & = b_2^0(W)+b_2^+(W), \\
\label{durf2}  4 p_g & = \mu(W) + \sigma(W) + b_1(Y).
\end{align}
If the singularity is Gorenstein, then
\begin{align}
\label{lau1} \mu(W) &  = 12 p_g+Z_K^2+ |\mathcal{V}|-b_1(Y), \\
\label{lau2} \sigma(W) & = - 8 p_g - Z_K^2 - |\mathcal{V}|,
\end{align}
where $|\mathcal{V}|$ is the number of vertices in the given resolution graph, and $Z_K$ is the canonical cycle of the resolution, defined by \eqref{adjunction}.
\end{prop}

The two equalities~\eqref{durf1} and~\eqref{durf2} are equivalent by \eqref{b20=b1}, and the formula~\eqref{lau2} follows from~\eqref{lau1} and~\eqref{durf2}. 
The first formulas were obtained by Laufer \cite{Laufermu} for the case of a hypersurface singularity, where one applies the Hirzebruch-Riemann-Roch theorem to compute and compare the analytic Euler characteristic of the closed smooth complex manifolds given by the compactification of the singular surface and the Milnor fiber in $\CC\PPP^3$, and a similar compactification of the resolution. The general case is based on a the same strategy together with a globalization result \cite{Loo} that allows to realize the smoothing in a compact space.  

It is important to note that the geometric genus is an {\em analytic} invariant: in general $p_g$ is {\em not} determined by the link $Y$ and depends on the analytic type of $(X, 0)$. In certain special cases, however, the $p_g$ is determined by topology and can be read off directly from the resolution graph. This is the case, for example, for rational singularities ($p_g=0$), or in the case of minimally elliptic singularities ($p_g=1$ and $(X,0)$ Gorenstein). In the general case, even though the value of $p_g$ can vary for the singularities with the same link, there are combinatorial upper bounds for $p_g$ in terms of the resolution graph~\cite{NemOk, Nem-book} (see also Section \ref{s:examples}):

\begin{prop}  \label{pg-bounds}  \cite[Section 6.8.B]{Nem-book}  For a fixed link $Y$, there exists a number $M>0$ such that $p_g\leq M$ for every  analytic singularity type $(X, 0)$ with link $Y$.
\end{prop}

We can now prove Theorem~\ref{bounds}.

\begin{proof}[Proof of Theorem~\ref{bounds}] This is an immediate corollary of Propositions~\ref{stip-bound}, \ref{durfee} and \ref{pg-bounds}. Indeed,
 $$
 2 \chi(W)+ 3 \sigma(W) = 5 b_2^+(W) - b_2^-(W) + 2 b_2^0(W) - 2b_1(W) +2,
 $$
and by Proposition~\ref{stip-bound} this quantity is bounded below. On the other hand, by Propositions~\ref{durfee},
~\ref{pg-bounds} and formulae~\eqref{bettiYW} and \eqref{b20}, we know that $b_1(W)$, $b_2^+(W)$ and $b_2^0(W)$ are both bounded above. Therefore, $b_2^-(W)$ must be bounded above as $W$ ranges over all possible Milnor fibers $W$. It follows that
the Euler characteristic and signature of all Milnor fillings of $(Y, \xi)$ must be also bounded.
\end{proof}

\begin{remark}
The following observation, which is a consequence of Proposition~\ref{stip-bound}, is due to Stipsicz~\cite{Stip1}:  if $b_2^+$ can be controlled for a family of Stein fillings of $(Y, \xi)$, then $b_2^-$ is also bounded, and therefore this family has bounded Euler characteristic and signature. For example, the principle applies if $(Y, \xi)$ admits a planar open book or is an L-space: then $b_2^+=0$ for gauge-theoretic reasons, and thus all fillings have bounded topology. As we observe, $b_2^+$ is uniformly bounded for Milnor fillings of a given link.
\end{remark}

\smallskip
\section{Proof of Theorem~\ref{thm:examples} and examples}\label{s:examples}

We start with the proof our theorem which will allow us to generate unexpected Stein fillings with arbitrarily large topology.

\begin{proof}[Proof of Theorem~\ref{thm:examples}] We use the results of Gay--Mark~\cite{GayMark} to construct an open book decomposition of the link $(Y, \xi)$. Guided by the dual resolution graph of the singularity, we build the page of the open book as follows. For an irreducible exceptional curve $E$ in the resolution, take a connected surface $S_E$ of genus $g(E)$ with $-a(E) - E \cdot E$ boundary components. The latter quantity is non-negative by hypothesis (i) of the theorem. Whenever two exceptional curves  intersect, we form the connected sum of the corresponding surfaces: in other words, connected sum necks correspond to the edges of the dual resolution graph. The resulting  connected sum of all surfaces $S_E$ gives the page $S$. The monodromy of the open book is the product of positive Dehn twists given by a boundary twist along each boundary component of $S$ and a twist along each connected sum neck. See Figure~\ref{openbk}.

\begin{figure}[htb]
	\centering
	\includegraphics[scale=.9]{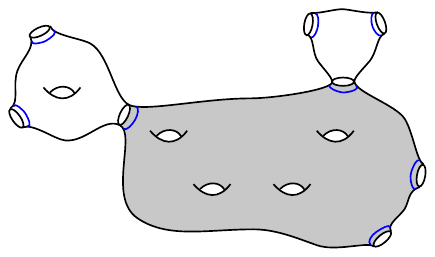}
	\caption{The Gay-Mark open book for the resolution of a non-rational singularity which admits infinitely many unexpected fillings. The monodromy is the product of positive Dehn twists
	along the curves shown in blue.
	The shaded subsurface of genus~$g=4$ with $b=4$ boundary components is the one where the arbitrarily long positive factorizations are supported.}
	\label{openbk}
\end{figure}

By construction and due to hypothesis (ii) of the theorem, the page of the open book has a subsurface $S_0$ of genus $g>2$; when the monodromy of the open book fixes this subsurface. The restriction of the monodromy to $S_0$ is the  boundary multitwist. Because for $b= - E_0 \cdot E_0$ we have  $b \leq 2g -4$ by assumption, results of~\cite{BMVHM}
yield arbitrarily long monodromy factorizations of the boundary multitwist on $S_0$. It follows that the monodromy
of the open book supporting $(Y, \xi)$ also has arbitrarily long factorizations --- and this is reflected in their $b_2(W)$ ---,
 so  corresponding Lefschetz fibrations give an infinite family of Stein fillings with unbounded Euler characteristic. By Theorem~\ref{bounds}, only finitely many of these can be realized by Milnor fibers.
 \end{proof}

In the remaining part of this section, we are going to describe some concrete Stein fillings of links of normal surface singularities that cannot be Milnor fibers by invoking Theorem~\ref{thm:examples}. To this end, we show how to find explicit values for the bounds of Propositions~\ref{stip-bound} and~\ref{pg-bounds}, and compute the relevant invariants in concrete examples.

\subsection{Geometric genus bounds} \label{subsect:geom-genus}
First, we recall the bounds for the geometric genus from~\cite[Section 6.8.B]{Nem-book}. These bounds are constructive and can be explicitly  computed in our examples. For a more self-contained exposition, we review the relevant material from~\cite[Section 6.8.B]{Nem-book}. 

Fix a good resolution $\tilde{X}$ of $(X, 0)$, with dual resolution graph $\Gamma$ whose vertices $v\in \Gamma$ correspond to exceptional curves $E_v$, $v\in \Gamma$. Let $g_v$ be the genus of $E_v$.

The lattice $L = H_2(\tilde{X}, \Z)$ is freely generated over $\Z$ by the classes of $E_v$, $v \in \Gamma$; its intersection form is negative definite. Let $L' = H^2(\tilde{X}, \Z) = H_2(\tilde{X}, \partial \tilde{X}, \Z)$ be the dual lattice; $L'$ can be identified with $\Hom (L, \Z)$. Extending the intersection form on $L$ to $L \otimes \QQ$, we can think of $\Hom (L, \Z)$ as the set of  elements $l' \in L\otimes \QQ$ such that $(l', l) \in \Z$.  Thus, we will identify $L'$ with a sublattice of $L \otimes \QQ$ and write its elements as rational linear combinations of the generators $E_v$.

Let $Z_K= -c_1 (\Omega_{\tilde{X}}^2) \in L'$ be the canonical cycle of the resolution.
The cycle $Z_K$ satisfies the adjunction relations
\begin{equation}\label{adjunction}
Z_K E_v = {E_v}\cdot E_v +2 - 2 g_v, \quad v \in \Gamma,
\end{equation}
and is in fact uniquely determined by these relations as a rational linear combination of classes $E_v$,
since the intersection form corresponding to $\Gamma$ is negative definite. It is known (see \cite[Example 6.3.4]{Nem-book})
that  all coefficients of $Z_K$ are positive if $\tilde{X}$ is a minimal good resolution and  if $(X, o)$ is {\em not} a rational double point.
For rational double points, the minimal resolution is good, $Z_K=0$, and  $p_g=0$.
Let $\lfloor Z_k \rfloor \in L$ be the integral cycle given by the floor of $Z_K \in L'$; 
the \emph{floor} of a rational cycle $Z=\sum r_v E_v$ is defined as $\lfloor Z \rfloor = \sum \lfloor r_v \rfloor E_v$, where $\lfloor r_v \rfloor$ stands for the greatest integer less than or equal to $r_v \in \QQ$.

In what follows, we will use short exact sequences of the form
\begin{align}
 \label{exact-seq1}
0 & \longrightarrow \OO_{\tilde{X}}(-Z) \longrightarrow \OO_{\tilde{X}} \longrightarrow \OO_{Z} \longrightarrow 0, \\
 \label{exact-seq2}
0 & \longrightarrow \OO_{F}(-E) \longrightarrow \OO_{F+E} \longrightarrow \OO_{E} \longrightarrow 0
\end{align}
for effective non-zero  divizors $Z$, $E$, and $F$ on $\tilde{X}$.
Here, as usual, $\OO_{F}(-E)$ stands for
$\OO_{\tilde{X}}(-E)|_F= \OO_{\tilde{X}}(-E)/\OO_{\tilde{X}}(-E-F)$, and $\OO_{\tilde{X}}(-Z)$ is the sheaf of holomorphic functions vanishing
to order at least  $n_i$ along each irreducible component $E_i$ of $Z=\sum_i n_i E_i$.
These exact sequences follow from the definition of the structure sheaves.
We will also use the fact that
\begin{equation}\label{vanish}
H^{\geq 2}(\OO_{\tilde{X}}(-Z))=0, \quad H^{\geq 2}(\OO_{F}(-E))=0,
\end{equation}
because higher cohomologies vanish for sheaves of holomorphic sections of line bundles on complex curves, and, more generally, on 1-dimensional complex spaces, and in the case of  $\tilde{X}$ it follows from the  \emph{Formal Function Theorem}, see e.g. \cite[Section 6]{Nem-book}.

\begin{lemma} The geometric genus of $(X, 0)$ can be expressed as
\begin{equation}\label{pg-ZK}
p_g=h^1(\tilde{X}, \OO_{\lfloor Z_K \rfloor}).
\end{equation}
\end{lemma}
\begin{proof}
Consider the short exact sequence of sheaves
$$
0 \longrightarrow \OO_{\tilde{X}}(-\lfloor Z_K \rfloor) \longrightarrow \OO_{\tilde{X}} \longrightarrow \OO_{\lfloor Z_K \rfloor} \longrightarrow 0.
$$
From the cohomological long exact sequence, we have
$$
\dots \longrightarrow H^1 ( \OO_{\tilde{X}}(-\lfloor Z_K \rfloor) ) \longrightarrow H^1( \OO_{\tilde{X}}) \longrightarrow H^1 (\OO_{\lfloor Z_K \rfloor}) \longrightarrow H^2 (\OO_{\tilde{X}}(-\lfloor Z_K \rfloor)) \longrightarrow \dots
$$
By~\eqref{vanish}, we get zero on the right; we also have
$H^1 ( \OO_{\tilde{X}}(-\lfloor Z_K \rfloor) )=0$ by  the Generalized Grauert-Riemenschneider Vanishing Theorem, see
~\cite[Section 6.4]{Nem-book}.  Therefore, $p_g=h^1(\OO_{\tilde{X}})= h^1(\OO_{\lfloor Z_K \rfloor})$.
\end{proof}

We can now estimate $h^1(\lfloor Z_K \rfloor)$ via a d\'evissage type argument. Consider a \emph{path}
 in the lattice $L$ from $0$ to $\lfloor Z_K \rfloor$, that is, take 
 a a sequence of integral cycles $\{l_i\}\in L$  such that
$$
l_0=0, \dots, \quad l_{i+1}=l_i+E_{v_i}, \dots \quad l_t= \lfloor Z_K \rfloor,
$$
where $E_{v_i}$ is the  basis element of $L$ corresponding to  some $v_i \in \Gamma$. (We will also
denote by $l_i$ the corresponding effective divisor in $\tilde{X}$.)
Writing $\OO_{l_i}= \OO_{\tilde{X}}|_{l_i}$, we have short exact sequences at each step, $i=0, \dots, t-1$,
$$
0 \longrightarrow  \OO_{E_{v_i}}(-l_i) \longrightarrow \OO_{l_{i+1}} \longrightarrow \OO_{l_i} \longrightarrow 0,
$$
and the corresponding long exact sequences in cohomology,
$$
\dots\longrightarrow H^0(\OO_{l_i} ) \longrightarrow H^1(\OO_{E_{v_i}}(-l_i) )
\longrightarrow H^1(\OO_{l_{i+1}} ) \longrightarrow H^1 (\OO_{l_i} ) \longrightarrow
H^2 (\OO_{E_{v_i}}(-l_i) ) \longrightarrow \dots.
$$
Since $H^2 (\OO_{E_{v_i}}(-l_i) )=0$, we get
\begin{equation}
\label{i-bound}
h^1(\OO_{l_{i+1}})- h^1(\OO_{l_i}) \leq h^1 (\OO_{E_{v_i}}(-l_i)).
\end{equation}
Here  $E_{v_i}$ is a smooth curve of genus $g_{v_i}$,
and $\OO_{E_{v_i}}(-l_i)$ is the sheaf of sections of a line bundle of
degree $d_i = - l_i \cdot E_{v_i}$ on $E_{v_i}$. Recall that for a line bundle $L$ of degree $d$
on a smooth curve $C$ of genus $g$  we have the following inequalities
$$h^1(L) \left\{
\begin{array}{ll}
=g-d-1 &\mbox{if} \ \ \  \ \ \ \ \ \ \ \ \ \  \  \ d\leq -1,\\
\leq g-\lceil d/2\rceil & \mbox{if} \ \ \ \ \ \ \ \ \  \ 0 \leq d\leq 2g-2,\\
=0 & \mbox{if}\ \ \   2g-1\leq d.
\end{array}\right.$$
Indeed, by Riemann-Roch,  $h^1(L)=h^0(L)+g-d-1$.  In particular, if $d\leq -1$ then $h^0(L)=0$ and $h^1(L)=g-d-1$.
If  $0\leq d\leq 2g-2$ then $h^0(L)\leq 1+\lfloor d/2\rfloor$ by Clifford's Theorem (see e.g. \cite[p. 107]{AlgCurves}).
 On the other hand, for any $d$,
$h^1(C, L)= h^0(C, L^{-1}\otimes K_C)$ by the Kodaira--Serre duality, where $K_C$ is the  canonical bundle of $C$ with degree $2g-2$.
Hence, if $d\geq 2g-1$, then ${\rm deg}(L^{-1}\otimes K_C)<0$ and  $h^0(L^{-1}\otimes K_C)=0$.

In particular, along our path,
$$
h^1 (\OO_{E_{v_i}}(-l_i)) \leq B_i := \left\{
\begin{array}{ll}
g_{v_i}-d_i-1 &\mbox{if} \ \ \ \  \ \ \ \ \ \ \ \ \ \ \ \ \ d_i\leq-1,\\
g_{v_i}-\lceil d_{i}/2\rceil& \mbox{if} \ \ \ \  \  \ \  \ \ \ \ \ 0\leq  d_i\leq 2g_{v_i}-2,\\
0 & \mbox{if}\ \ \  2g_{v_i}-1\leq d_i.
\end{array}\right.$$

Now add the inequalities~\eqref{i-bound} for all $0\leq i<t$ to get
\begin{equation}\label{sum-bound}
p_g= h^1(\lfloor Z_K \rfloor) \leq \sum_{i=0}^{t-1} B_i.
\end{equation}
Note that the resulting bound \eqref{sum-bound} depends only on the resolution graph but not on the analytic structure.
The bound from the right-hand side of \eqref{sum-bound} depends also on the choice of the path, with special well--chosen paths
we can realize sharper bounds.

The bound \eqref{sum-bound} is not very sharp in general; this weakness usually occurs when some of the genera $g_v$ are non-zero.
However, if the link is a rational homology sphere, then in many cases the bound  \eqref{sum-bound} is optimal, that is, for a well-chosen path,
the topological right-hand side can be realized as the geometric genus of a (very special) analytic structure. See subsection 11.3.A  or Corollary 11.6.11 of \cite{Nem-book}.

\begin{example}
Consider the case where the minimal resolution has a unique exceptional curve $E$ of genus $g$ with self-intersection $-b$, where $g, b\geq 1$.  This is the simplest case of Theorem~\ref{thm:examples}, where for $b\leq 2g -4$ we are able to generate infinitely many Stein fillings. From~\eqref{adjunction}, we have
$$
Z_K = \frac{2g-2+b}{b}\, E, \quad \lfloor Z_K \rfloor = \left(\left\lfloor \frac{2g-2}b \right\rfloor +1\right) E.
$$
The upper bound as in $\eqref{sum-bound}$ is then obtained from the path
$$
0, E, 2E, \dots, \left(\left\lfloor \frac{2g-2}b \right\rfloor +1\right) E.
$$
Then we apply  \eqref{sum-bound}, where at $i$-th stage we consider the line bundle of degree
$d_i=-(i E)\cdot E = bi $ on $E$, $0\leq i\leq \lfloor(2g-2)/b\rfloor$.
Note that $d_i=ib \leq  2g-2$. Therefore,
\begin{equation}\label{pg<4g-squared}
p_g \leq
\sum_{i=0}^{\lfloor(2g-2)/b\rfloor} \Big( g-\Big\lceil\frac{ib}{2}\Big\rceil\Big).
\end{equation}
The reader is invited to compute the bound for some specific values  of $b$ and $g$.  For example,
 if $b=1$ then we obtain $p_g\leq g^2$.
  If    $b=2$  then
  $p_g\leq  (g^2+g)/2$.

When $g>3$ and $b=2g-4$  (the maximum value of $b$ where we will be able to produce infinitude of fillings),
the sum \eqref{sum-bound} comes from the path $0, E, 2E$ and has only two terms: $B_0=g$ and $B_1=2$.  Hence
\begin{equation}\label{pg<2g+2}
p_g \leq B_0 +B_1 = g+2.
\end{equation}
If $g=3$ and $b=2g-4=2$ then $p_g\leq B_0+B_1+B_2=3+2+1=6$.

A straightforward computation shows that for any such $b,g\geq 1$ the following inequality holds:
\begin{equation}\label{pg<4g-squared-2}
p_g \leq \sum_{i=0}^{t-1} B_i \leq g^2.
\end{equation}
\end{example}

\subsection{Computing topological invariants} To constrast these bounds with Proposition~\ref{durfee} and pinpoint specific Stein fillings that cannot
come from Milnor fibers, we compute the topological invariants for a particular family of Stein fillings we will describe explicitly, through allowable Lefschetz fibrations they are equipped with. We focus on the examples where the link is an $S^1$-bundle over a surface, with its Boothby--Wang structure. As noted earlier, in this case, we have a compatible open book with page $\Sigma_g^b$, a compact genus $g$ surface with $b$ boundary components, and the monodromy is the boundary multitwist $\Delta:=t_{\delta_1} \cdots t_{\delta_{b}}$. We assume that $g \geq 3,  b \geq 1$, and $b \leq 2g-4$. We will describe the allowable Lefschetz fibrations on our Stein fillings via their monodromy factorizations coming from~\cite{BMVHM} and compute the invariants of these fillings.

For a product  $P$ of positive Dehn twists in $\M(\Sigma_g^b)$, let the \emph{length} $|P|$ be the number of Dehn twist factors in $P$.  Recall that for any Dehn twist $t_c$ and $ \phi \in \M(\Sigma_g^b)$,  the conjugate of the Dehn twist $t_c^\phi=\phi \, t_c \, \phi^{-1}=t_{\phi(c)}$.  In particular,  for a product $P$ of positive Dehn twists,  $P^\phi$ is also a product of positive Dehn twists with $|P^\phi|=|P|$.

We are going to describe factorizations of the multitwist $\Delta$ into arbitrarily large number of positive Dehn twists in $\M(\Sigma_g^b)$.  We begin with the maximal case $b=2g-4$, where in our factorizations, each Dehn twist curve $c$ will satisfy the following additional property: under an embedding of $\Sigma_g^b \hookrightarrow \Sigma_g^{b'}$,  the curve $c \subset \Sigma_g^b$ maps to a homologically essential curve in $\Sigma_g^{b'}$, for any  $b' \leq b$.  (Clearly, this then holds  for \emph{any} embedding.) One can view the embedding as capping off some of the boundary components of
$\Sigma_g^b$ with discs to get $\Sigma_g^{b'}$.

\begin{figure}[htb]
	\centering
	\includegraphics[height=155pt, trim=0 -10 0 10]{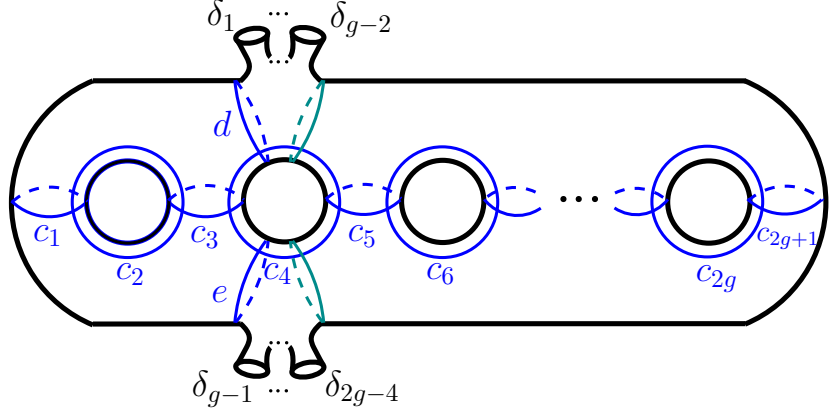}
	\caption{The curves $c_i, d, e$ on $\Sigma_g^{2g-4}$ with boundary components $\delta_j$.}
	\label{fig:curves1}
\end{figure}

\begin{figure}[htb]
	\centering
	\includegraphics[height=110pt]{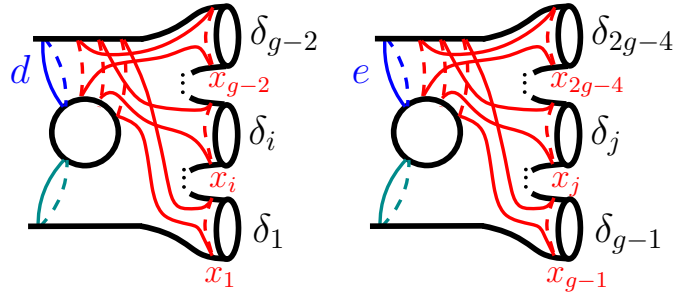}
	\caption{The curves $x_j$ drawn on the two subsurfaces of $\Sigma_g^{2g-4}$ containing all the boundary components.}
	\label{fig:curves2}
\end{figure}

Let $c_i,  d,  e,  x_j,  \delta_j$ be the curves on $\Sigma:=\Sigma_g^{2g-4}$ as shown in Figures~\ref{fig:curves1},~\ref{fig:curves2}  and let $d_k,  e_k, f_k$ be the curves:
\begin{align*}
& d_k=t_{c_{k-3}}^{-1}t_{c_{k-2}}^{-1}t_{c_{k-1}}^{-1}(c_k)\\
& e_k=t_{c_{k-3}} t_{c_{k-2}} t_{c_{k-1}}(c_k) & \\
& f_k=t_{c_{k-5}}t_{c_{k-4}}t_{c_{k-3}}t_{c_{k-2}}t_{c_{k-1}}(c_k)&
\end{align*}
whenever the indices are well-defined.  With these in mind,  we consider the following products of positive Dehn twists:

\begin{align*}
&\psi:= t_{c_4} t_{c_3} t_{c_2} t_{c_1}  t_{c_5} t_{c_4} t_{c_3} t_{c_2}  t_{c_6} t_{c_5} t_{c_4} t_{c_3}  t_{c_7} t_{c_6} t_{c_5} t_{c_4}  \\
&A_m:= (t_{c_4} t_{c_3} t_{c_2} t_{c_1} t_{c_1} t_{c_2} t_{c_3} t_{c_4} t_{c_4} t_d t_{c_3} t_{c_4})^{t_{c_3}^{-m}t_e^m}
((t_{c_1} t_{c_2} t_{c_3})^2 t_{c_2} t_{c_1} t_{c_3} t_{c_2})^m,\text{  } m \in \N\\
&B:=t_{(t_{c_4} t_d t_{c_3} t_{c_4})^{-1}(c_5)} t_{c_2} t_{(t_{c_5}t_{c_4})^{-1}(c_6)} t_d t_{t_{c_4}^{-1}(c_3)} t_{c_7} t_{(t_{c_5}t_{c_4})^{-1}(c_6)} t_d t_{t_{c_4}^{-1}(e)} \,  t_{c_5} t_{c_3} t_{c_4} t_{c_2} t_{c_3}  \\
&C:= (t_{d_{10}} t_{d_{11}} \cdots t_{d_{2g+1}}) \, (t_{e_{2g+1}} \cdots t_{e_9} t_{e_8}) \\
&D:=D_0^2 \, (t_{x_1} t_{x_2} \cdots t_{x_{g-2}}) \,  (t_{x_{g-1}} t_{x_g} \cdots t_{x_{2g-4}}) \,  (t_{f_9} t_{f_8} t_{f_7} t_{f_6} )
\end{align*}

where in the last product,  $D_0:=t_{c_1} t_{c_2} t_{c_3}$ for even $g$ and $D_0:=t_{c_3} t_{c_2} t_{c_1}$ for odd $g$.

It was shown in \cite[Proof of Theorem~9]{BMVHM} that for $\Sigma=\Sigma_g^{2g-4}$ the following hold:
\begin{equation}\label{eq:monodromy}
P_m:=A_m B \,C \,D^\psi = \Delta \ \ \ \text{in } \M(\Sigma) \ \  \text{ for all } m \in \N \, .
\end{equation}
Counting the Dehn twists in the factorization, we see that $|P_m|=6g+10m+18$.
Observe that as we cap off $\partial \Sigma$ with discs,  each one of the Dehn twist curves of the positive factorization $P_m$ of $\Delta$ maps to a non-separating curve in $\Sigma_g$; so its image remains homologically essential under any embedding $\Sigma=\Sigma_g^{2g-4} \hookrightarrow \Sigma_g^{b'}$ for $1 \leq b' \leq 2g-4$, satisfying the additional property we asked for.  In particular,  the Lefschetz fibration $X_m$ over the disc with monodromy factorization~\eqref{eq:monodromy} is \emph{allowable},  and  yields a Stein filling of the  contact boundary $Y:=\partial X_m$,  taken with the contact structure compatible with the fixed boundary open book.

Next, we calculate the homology groups of the Stein fillings $X_m$ given by~\eqref{eq:monodromy}.

\noindent \underline{\textit{First homology}}. We will  show that $H_1(X_m)=0$ for each $m \in \N$.  We have $G:=H_1(X_m) = H_1(\Sigma) \,  / \, N$ where $N \leq  H_1(\Sigma)$ is generated by the homology classes of the vanishing cycles of the monodromy factorization  of $X_m$.  Note that the homology classes of the curves $\{ c_i, \delta_j\}$,  taken with auxiliary orientations,  generate $H_1(\Sigma)$. So to prove our claim, it suffices to show that there is a collection of monodromy curves $ \{v_k\}$  whose homology classes generate those of  all $c_i$, $\delta_j$ in $H_1(\Sigma)$, or equivalently,
by arguing that the images of $[c_i], [\delta_j]$ in the quotient  $H_1(\Sigma) \,  / \, \langle \{[v_k]\} \rangle$ will all be trivial.  By a slight abuse of notation, whenever $x \in H_1(\Sigma)$ can be seen to be generated by $ \{[v_k]\}$, we are going to say that  $x=0$ in $G$.

Since the order of the Dehn twists in the monodromy factorization here  is irrelevant,  we can carry out our calculation for both odd and even genus $g$ at the same time,  as the two cases only differ by only the order of the Dehn twists in the subproduct $D_0$.  Also,  without explicitly mentioning it after the first few instances, we are going to repeatedly appeal to the Picard--Lefschetz formula whenever a monodromy curve is expressed as the image of a curve under a mapping class that is expressed as a product of (possibly negative) Dehn twists.

From the monodromy curves $c_1,  c_2, c_3$ of  $A_m$ and $c_4, c_5, c_7, d$ of $B$,  we immediately have $[c_i]=0$ for $i=1, 2,3,4,5$ and $7$, and $[d]=0$ in $G$. (When $m=0$, the curves $c_1, c_2, c_3$ are in the now trivially-conjugated factor of $A_m$.)
Moreover,  by Picard--Lefschetz, the homology class of the monodromy curve  $(t_{c_5}t_{c_4})^{-1}(c_6)$ of $B$  is equal to $[c_6] + x$ in $H_1(\Sigma)$ for some $x \in \langle [c_4], [c_5]\rangle$,  and since $x=0$ in $G$,  we conclude that $[c_6]=0$ in $G$, too.  Similarly,   the monodromy curve $t_{c_4}^{-1}(e)$ of $B$ yields $[e]=0$ in $G$.  Next,  applying Picard--Lefschetz to each monodromy curve $e_i=t_{c_{i-3}} t_{c_{i-2}} t_{c_{i-1}}(c_i)$ in $C$,  we see that $[e_i]=[c_i]+x$ for some $x \in \langle [c_{i-3}], [c_{i-2}], [c_{i-1}] \rangle$.  Since we already have $[c_i]=0$ for all $i=1, \ldots, 7$,  using the monodromy curves $e_8, \ldots , e_{2g+1}$,  we inductively see that $[c_i]=0$ also for $i=8, \ldots, 2g+1$.

Therefore, we have $[c_i]=0$ for all $i=1, \ldots, 2g+1$ and $[d]=[e]=0$ in $G$.
Now,  the homology classes of the monodromy curves $x_j$ in $D^\psi$ are all of the form $[x_j] +x$ for $x \in \langle \{[c_i]\} \rangle$, since the conjugating mapping class $\psi \in \langle \{[c_i]\} \rangle$.  So we also get $[x_j]=0$ for all $j=1, \ldots, 2g-4$.  Orienting $\{x_j\}$ and $\{\delta_j\}$ suitably, we have
\[
[x_j]=
\begin{cases}
    [d]+[\delta_j],& \text{if } j=1, \ldots , g-2, \\
    [e]+[\delta_j],& \text{if } j=g-1, \ldots, 2g-4.
\end{cases}
\]
Since $[x_j]=[d]=[e]=0$,  it follows that $[\delta_j]=0$ in $G$ for all $j=1 ,\ldots, 2g-4$.

Hence, all the generators $[c_i]=[\delta_j]=0$ in $G$, so $H_1(X_m)=0$.\footnote{
In fact,  a similar calculation,  where  we keep track of the based homotopy classes of the monodromy curves under conjugations (instead of invoking Picard--Lefschetz for a fast track),  shows that $\pi_1(X_m)=1$ for any $m \in \N$; cf. ~\cite[Proof of Theorem~B]{ArabadjiBaykur}.}

\noindent \underline{\textit{Second homology}}. First of all, we have $H_1(X_m,  Y)=0$ since $H_1(X_m)=0$ and $Y$ is connected.  By the universal coefficients theorem and Poincar\'{e}-Lefschetz duality,  we conclude that $H_2(X_m)$ has no torsion.

Since $|P_m|=6g+10m+18$ is the number of nodes $\ell$  in $X_m$, we calculate the Euler characteristic of $X_m$ as
\[ \chi(X_m)=\chi(D^2)\,  \chi (\Sigma_g^{2g-4}) + \ell=6-4g+(6g+10m+18)=2g +10m+24 \, .
\]
Since $X_m$ admits a $2$--handlebody,   we  also have
\[
\chi(X_m)= b_0(X_m) - b_1(X_m)+b_2(X_m) = 1+ b_2(X_m) \, .
\]
It follows that $H_2(X_m) \cong \Z^{2g +10m+23}$.

By~\eqref{bees}, we have
$b_2^0(X_m)=b_1(\partial Y)-b_1(X_m)= b_1(Y)$.  Here $Y$ is diffeomorphic to an $S^1$--bundle $\Sigma_g$ with Euler number $b=2g-4 \neq 0$,  so $H_1(Y) \cong \Z^{2g} \oplus \Z_b$.  Thus,  $b_2^0(X_m)=2g$,
so the rank of the maximal non-degenerate subspace of $H_2(X_m)$ with respect to the intersection form is $b_2(X_m)-b_2^0(X_m)=10m+23$.

In order to calculate the signature of $X_m$, let us first calculate the signature of the associated Lefschetz \emph{pencil} $\widehat{X}_m$ prescribed by the (same) monodromy factorization~\eqref{eq:monodromy}. We calculate the latter following the approach in \cite{EndoNagami}: Take the infinite presentation of $\M(\Sigma_g^b)$ with generators as all Dehn twists along homotopically essential curves, and then $\sigma(\widehat{X}_m)$ is the algebraic sum of the signatures of the relators involved in the derivation of~\eqref{eq:monodromy} from the trivial word. As we  examine the derivation of the relation ~\eqref{eq:monodromy} in \cite{BMVHM}, we see that the only relators with nontrivial signature contributions employed here are odd chain relators (of different lengths and some are inverses) and a pair of daisy relators each of which are derived from $(g-3)$ lantern relators. Denoting a length $k$ chain relator by $C_k$ and the lantern relator by $L$, we have:

\begin{eqnarray*}
\sigma(\widehat{X}_m)&=&\sigma(C_{2g+1})
+\sigma(C_{2g-3}^{-1})+(g-2)\sigma(C_{3}^{-1}) +m\sigma(C_3) + (2g-6) \sigma(L)
\\
&=&  -2g(g + 2)+ 2(g-2)g+(g-2-m)\,6+(2g-6)
\\
&=& -6m-18  \, .
\end{eqnarray*}
As $X_m$ is the complement of a regular fiber of self-intersection $b=2g-4>0$ in the pencil $\widehat{X}_m$, we have
\[
\sigma(X_m)=\sigma(\widehat{X}_m)-1=-6m-19
\, .
\]

Combining  $b_2^+(X_m)+b_2^-(X_m)=10m+23$ and
$b_2^+(X_m)-b_2^-(X_m)=-6m-19$ we then get
$b_2^+(X_m)=2m+2$ and $b_2^-(X_m)=8m+21$.

Moreover, we can easily see that $X_m$ has  odd intersection form.  This follows from the observation that the monodromy factorization~\eqref{eq:monodromy} is Hurwitz equivalent to another factorization of the form
\begin{equation*}
 t_{c_1}t_{c_3} t_d P'_m = \Delta  \ \ \ \text{in } \M(\Sigma) \ \  \text{ for all } m \in \N \,
\end{equation*}
which is simply obtained by conjugating away all Dehn twists other than $t_{c_1}, t_{c_3}, t_d$ that appear in the original factorization.  Clustering these three vanishing cycles on the same fiber, we get another Lefschetz fibration whose total space is still $X_m$, but now it has a reducible fiber component $\Sigma'  \cong S^2$, which is obtained by pinching these three vanishing cycles.  As the fiber framing of each Lefschetz $2$--handle is $-1$, the surface $\Sigma'$ has self-intersection $-3$.

Hence, the indefinite, odd intersection form of $X_m$ is
\[
Q_{X_m} \cong (2m+2)\langle 1 \rangle \oplus (8m+23) \langle -1 \rangle \oplus (2g) \langle 0 \rangle \, .
\]

\smallskip
More generally, let $X_{g,b, m}$ denote the allowable Lefschetz fibration of genus $g \geq 3$ with monodromy factorization
\begin{equation*}
P_{g, b,m}:=A_m B \,C \,D^\psi = \Delta \ \ \ \text{in } \M(\Sigma_g^{b}) \ \  \text{ for all } m, b \in \N\text{ with
 } 1 \leq b \leq 2g-4   \, ,
\end{equation*}
where for simplicity we use the same letters for the images of the products of Dehn twists in~\eqref{eq:monodromy}  under the homomorphism induced by the embedding $\Sigma_g^{2g-4} \hookrightarrow \Sigma_g^b$ and $\Delta$ is the boundary multitwist on $ \Sigma_g^b$. These are Stein fillings of the contact $3$--manifold $Y_{g,b, m}$ taken with the contact structure compatible with the boundary open book with monodromy $\Delta \in \M(\Sigma_g^b)$. From our calculations above for the maximal case $X_m=X_{g, 2g-4, m}$ we easily conclude that $H_1(X_{g, b, m})=0$ as well and
\begin{equation} \label{eq:intform}
Q_{X_{g,b,m}} \cong (2m+2)\langle 1 \rangle \oplus (2g-b+8m+17) \langle -1 \rangle \oplus (2g) \langle 0 \rangle \, .
\end{equation}

\medskip
\begin{example} For the fillings $X_{g,b,m}$ of $Y_{g,b}$ we have discussed above, we can now contrast the invariants we calculated with the bounds for geometric genus and the formulas of Proposition~\ref{durfee}. 

We have $b_2^0(X_{g,b,m})=2g$ and
$b_2^+(X_{g,b,m})=2m+2$, thus if $X_{g,b,m}$ is a Milnor fiber, we must have
$2p_g=2g+2m+2$. Since by~\eqref{pg<4g-squared-2} we know that $p_g \leq g^2$,  $X_{g,b,m}$ cannot be a Milnor filling for
$m>g^2-g-1$. If $g>3$ and $b=2g -4$, then $p_g \leq g+2$ by~\eqref{pg<2g+2} so taking $m>1$ will ensure that  $X_{g,2g-4,m}$ is not a Milnor fiber.
If $g=3$ and $b=2$ then  $m>2$ will ensure that  $X_{g,2g-4,m}$ is not a Milnor fiber.

It might be worth noting that, alternatively, we can avoid the signature calculation by combining ~\eqref{c12}, ~\eqref{durf2} and~\eqref{pg<4g-squared-2} to show that
$$
\chi(W) \leq 12 g^2 - 2g + 2b -1
$$
for all Milnor fillings, and therefore  $X_{g,b,m}$ cannot be a Milnor fiber when $10m> 
12g^2 -6g+3b -21$. Similarly, for
the case $b=2g -4$ this argument gives a bound $\chi \leq 12 g^2+2g-9 $, and it follows that  $X_{g,2g-4,m}$ is an unexpected filling for  $10m>12 g^2-33$. This simpler calculation gives a larger bound for $m$ compared to the bound obtained in the previous paragraph, but it has the advantage of generalizing easily to handle other examples covered by Theorem~\ref{thm:examples}.
\end{example}

 \section{Unexpected fillings for minimally elliptic singularities} \label{s:min-elliptic}

 Minimally elliptic singularities are Gorenstein singularities of geometric genus $p_g=1$. They also admit a topological definition, in terms of the minimal resolution $\tilde{X}$, as follows.   Let $L$ be the homology lattice of the $4$--manifold with boundary given by the resolution, with the basis given by the homology classes of the irreducible exceptional curves $E_v$, $v \in \Gamma$. We use the same notation as in Subsection \ref{subsect:geom-genus},
 but we emphasize that one needs to take the minimal (not minimal good) resolution for the topological characterization (so $\Gamma$ will just stand for the set of vertices). We recall the definition of Artin's fundamental cycle of the singularity, denoted $Z_{min} \in L$. (This definition will hold in any resolution.) Consider the set of divisors
$$
\{ Z= \sum_{\substack{v \in \Gamma, \\ m_v \geq 0}} m_v E_v \in L\ | \  Z \cdot E_v \leq 0 \text{ for all } E_v \}.
$$
This set has a partial order, defined by    $\sum m_v E_v    \geq \sum n_v E_v$   if $m_v \geq n_v$ for all $v$. There is a unique nonzero minimal element $Z_{min}$
with respect to this partial order, called the \emph{fundamental cycle}. Since the exceptional divisor is connected, we must have $Z_{min} \geq \sum_{v \in \Gamma} E_v$.

The Riemann--Roch function  $\chi: L \to \Z$ is given by 
 $$
\chi(D) = (-D \cdot D + D\cdot Z_K)/2,
$$
where  $Z_K= -c_1 (\Omega_{\tilde{X}}^2)$ is the canonical class. (Formula~\eqref{adjunction} is valid whenever all the irreducible exceptional curves $E_v$ are smooth, which will be the case for all our examples, although may not hold in general.)   A cycle $C\in L$, $C>0$  is called {\em minimally elliptic} if $\chi(C)=0$,  and  $\chi(l)>0$  for each $l\in L$ such that $0<l<C$. The singularity is minimally elliptic if $Z_{min}$ is a minimally elliptic cycle.\footnote{Equivalently, one can require that $\chi(Z_{min})=0$ and the graph $\Gamma$ satisfy the following property: if $\Gamma'$ is an arbitrary connected subgraph of $\Gamma$, with the same weights, then $\Gamma'$ supports a rational singularity; recall that rationality can also be detected from the resolution graph via Laufer sequences.}

 The topological characterization of minimally elliptic singularities implies  that every analytic singularity of the same topological type will be minimally elliptic. This property is important to us, as it allows to control all Milnor fillings for the given link even if they come from different analytic singularities. (For example, the graph in Figure~\ref{min-el-fam} has a valence 4 vertex when $n=1$, 
 therefore the corresponding topological type admits infinitely many  compatible analytic structures.) In particular, all compatible analytic types are Gorenstein and share the same $p_g=1$. If the link is a rational homology sphere, then by~\eqref{lau1}, \eqref{lau2}, any Milnor fiber $W$ of a minimally elliptic singularity satisfies
 \begin{equation} \label{lauf-min-ell}
 b_2^+(W)=2.
 \end{equation}
 Francesco Lin showed that all non-negative definite Stein fillings have the same intersection forms as Milnor fibers, in particular, they also satisfy ~\eqref{lau1}, \eqref{lau2}.

 \begin{prop}\label{prop:lin} ~\cite{Lin} Let $(Y, \xi)$ be the contact link of a minimally elliptic surface singularity with $b_1(Y)=0$.
 Then all Stein fillings of $(Y, \xi)$ have an even intersection form with $b_2^+=2$ and $b_2^-=8h+10$ for some fixed $h \in \Z$.
 \end{prop}
\begin{proof} Lin proves the result in a more general context, for contact 3-manifolds $(Y, \xi)$ whose reduced monopole Floer homology has rank 1 in the $Spin^c$ structure $\mathfrak{s}_{\xi}$ induced by $\xi$, where $\mathfrak{s}_{\xi}$ is required to be self-conjugate. The hypothesis on reduced 
homology is satisfied for links of minimally elliptic singularities due to isomorphisms between the monopole Floer and Heegaard Floer homologies (see \cite{Ta, CGH, KLT} and their sequels) and between Heegaard Floer homology and lattice cohomology  \cite{Nem-plumbed, zemke2023equivalence}. This is because the reduced lattice cohomology is known to have rank 1 in the canonical $Spin^c$ structure (and zero in all other $Spin^c$ structures). When orientations and grading shifts are taken into account, one sees that the minimally elliptic links are of ``Type II'' in the terminology of~\cite{Lin}. The canonical $Spin^c$ structure on $Y$ can be seen to be self-conjugate  because it is the unique 
$Spin^c$-structure with non-vanishing reduced homology.
 \end{proof}

Lin's results do not extend to negative definite fillings, and we will indeed produce such unexpected  Stein fillings with intersection forms different than that of the minimal resolution. By~\eqref{lau1}, \eqref{lau2}, the Euler characteristic and signature are the same for all the Milnor fibers. The minimal resolution may have different invariants, but it suffices to construct three or more Stein fillings with different invariants to find an unexpected filling. Below, we produce these as allowable Lefschetz fibrations, using monodromy factorization.

\subsection{An infinite family of minimally elliptic links}
Consider the family of minimally elliptic singularities prescribed by the dual resolution graphs in Figure~\ref{min-el-fam}, 
with $n=2k+1$ vertices in the $(-2)$-chain. These are of the form $A_{n, ***}$ in Laufer's notation \cite{laufer:minimally_elliptic}.

 \begin{figure}[htb]
	\centering
	\includegraphics[scale=1]{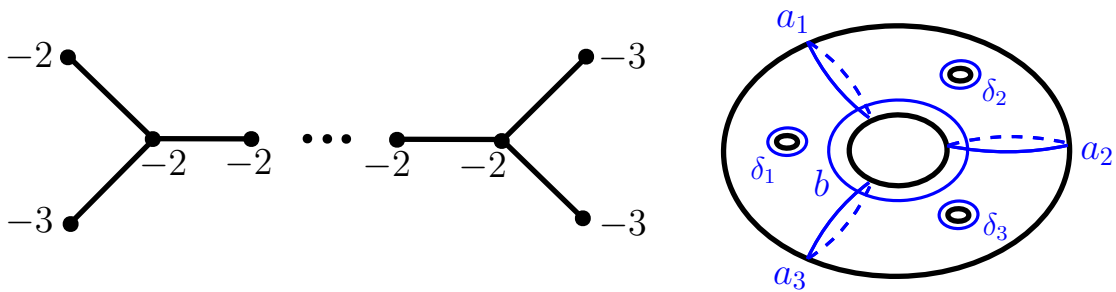}
	\caption{Minimally elliptic singularities with unexpected fillings.}
	\label{min-el-fam}
\end{figure}

The fundamental cycle here is $Z_{min}= Z_K=(1, 1, 2, 2, \dots, 2, 2, 1, 1)$, where $Z_{min}$ has coefficient $1$ on each end vertex of the graph and coefficient $2$ on each vertex in the chain of the $n$ vertices with weights $-2$. 

To get an open book decomposition of the link, we use Bhupal's algorithm in \cite{Bhu1, Bhu2}. The link $Y$ is obtained by plumbing circle bundles over spheres, as dictated by the resolution graph. Bhupal's procedure gives an open book whose binding is the union of a certain number of the fibers of these circle bundles, according to the numerical data from the intersection form of the resolution and the fundamental cycle. By~\cite{NeuPich, Pich}, this open book arises from a Milnor-type fibration associated to a holomorphic function vanishing at the isolated singular point of a germ of an appropriate analytic surface $(X,0)$, such that $Y$ is the link of $(X, 0)$. Uniqueness results of \cite{CNPP, CPP} then imply that the open book is compatible with the canonical contact structure on $Y$.

In Bhupal's algorithm, the page of the open book  built out of subsurfaces that correspond to the individual exceptional curves. (In our situation all exceptional curves have $g=0$, although sometimes the algorithm works more generally.) When the \emph{good vertex} condition~\eqref{goodvert} is satisfied, each 
subsurface has the same genus as the corresponding exceptional curve, with the number of holes determined by the Euler number and the valency of the vertex in the graph, and these subsurfaces form the page as dictated by the graph, see also \cite{GayMark, EO}.  
When we have a \emph{bad vertex}, that is a vertex violating the condition~\eqref{goodvert}, the subsurfaces are coverings of the exceptional spheres with punctures; the number of sheets is determined by the corresponding coefficient of $Z_{min}$. In this construction, the monodromy of the resulting open book is presented as a product that may contain fractional multitwists. Under favorable circumstances, the latter can be expressed as a product of positive Dehn twists. For further details of the algorithm, we refer the reader to \cite{Bhu1, Bhu2}.

Applying this algorithm, we get a compatible open book on the link $(Y_k, \xi_k)$ with monodromy\footnote{This expression is slightly different from the particular case contained in \cite{Bhu2}; as confirmed to us by Bhupal, the monodromy in~\cite[Case 63]{Bhu2} is incorrect due to a small notational mistake.}  
\begin{equation}\label{eq:inf-elliptic-monodromy}
\phi= t_{\delta_1} t_{\delta_2} t_{\delta_3}  t_{a_1}^k t_{a_3}^k t_{a_1} t_{a_3} t_{b} t_{a_2} t_{a_3} t_{b} \ \ \text{ in } \ \M(\Sigma_1^3) \, ,
\end{equation}
where the monodromy curves are as shown in Figure~\ref{min-el-fam}.   

For each $n$, we can produce three Stein fillings $\{X_i\}$ with different second Betti numbers. Each Stein filling will have $H_1(X_i)=0$, as easily seen by the monodromy curves generating the first homology of the page in each case, so $b_2(X_i)=\chi(X_i)-1$ is determined by the number of Dehn twist curves in the factorization.

The filling $X_1$ is given by the allowable Lefschetz fibration with monodromy factorization as above. We have $b_2(X_1)= n+4$,  which matches the $b_2$ of the minimal resolution. Next is the filling $X_2$, given by the allowable Lefschetz fibration whose monodromy factorization is obtained from~\eqref{eq:inf-elliptic-monodromy} by the \emph{star relation} (the elliptic pencil monodromy with three base points): we  replace $t_{\delta_1} t_{\delta_2} t_{\delta_3}$ factor by $(t_{a_1} t_{a_2} t_{a_3} t_{b})^3$. Then $b_2$ increases by $9$, so $b_2(X_2)= n+13$, which matches the $b_2$ of a Milnor fiber (here we computed $Z_K^2=-3$ and used \eqref{lau1}). Finally, we obtain the filling $X_3$ by applying a lantern substitution to~\eqref{eq:inf-elliptic-monodromy}: we replace the Dehn twists around the curves $\{a_1, a_3, \delta_2, \delta_3\}$ by three other positive Dehn twists. Here the presence of the boundary component $\delta_1$ ensures that all three of the new curves are also homologically essential on $\Sigma_1^3$. The corresponding Stein filling has $b_2(X_3)=n+3$ and thus cannot be a Milnor filling or the resolution.

\smallskip
\subsection{Triangle singularities} \label{sec:triangle} A special class of minimal elliptic singularities are the \emph{triangle singularities} $D_{p,q,r}$, also known as Dolgachev singularities. 
These are characterized by three rational curves transversally intersecting at one point, with a resolution graph as shown in Figure~\ref{fig:triangle}, where $\frac{1}{p}+\frac{1}{q}+\frac{1}{r} <1$ and  $r \geq q \geq p \geq 2$. 

 \begin{figure}[htb]
	\centering
	\includegraphics[scale=1]{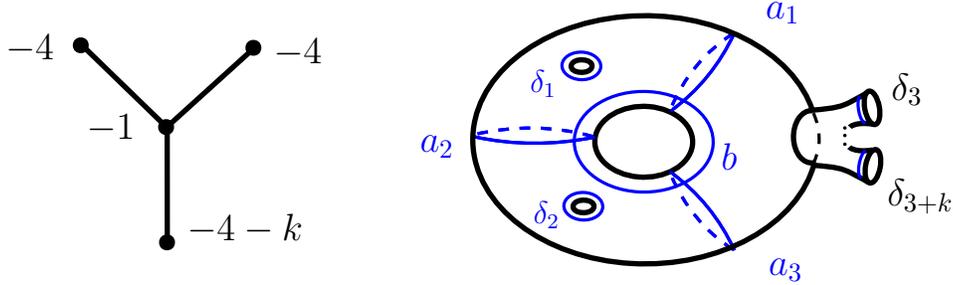}
	\caption{A family of smoothable triangle singularities and the monodromy curves for the open book on the link. }
	\label{fig:triangle}
\end{figure}

 A triangle singularity $D_{p,q,r}$ is known to be smoothable if and only if $p+q+r \leq 22$ and $(p,q,r) \neq (2,10,10)$ \cite{Looijenga:triangle, Pinkham2}. We consider the family $D_{4,4,4+k}$, for $k=0, \ldots, 6$,  where $k=0$ is the hypersurface singularity $\{x^4-y^3-z^3=0\} \subset \mathbb{C}^3$.  
 
 Label the vertices in the resolution graph in Figure~\ref{fig:triangle} so that the central vertex comes first and the one with weight $4+k$ comes last.  The fundamental cycle for it is $Z= (3, 1, 1, 1)$.  Applying Bhupal's algorithm, we see that the link $(Y_k, \xi_k)$ is supported by an open book with monodromy
\begin{equation}\label{eq:triangle-monodromy}
\phi= t_{\delta_1} t_{\delta_2} \cdots t_{\delta_{3+k}} t_{a_1} t_{a_2} t_{a_3} t_b  \ \ \text{ in } \ \M(\Sigma_1^{3+k}) \, ,
\end{equation}
where the monodromy curves are as shown in Figure~\ref{fig:triangle}.  Here the subsurfaces that make up the page $\Sigma_1^{3+k}$ of the open book are $\Sigma_1^3$ corresponding to the central vertex (as  a 3-fold cover of $\Sigma_0^6$ in the algorithm), two copies of $\Sigma_0^2$ corresponding to  the vertices with weights $(-4)$, and a $\Sigma_0^{k+2}$ corresponding to the last vertex. The resulting page is $\Sigma_1^{3+k}$. The algorithm yields a monodromy expression with $\frac{2\pi}{3}$--fractional multitwist along the boundary components of the subsurface $\Sigma_1^3$, which can be rewritten as shown using the star relation as above.

Once again we obtain three Stein fillings for each link $(Y_k, \xi_k)$ with distinct second Betti numbers. As before, we will easily see that that $H_1(X_i)=0$ and compute $b_2$ from the number of Dehn twist curves in the factorization. 

Let the filling $X_1$ be given by the allowable Lefschetz fibration with monodromy factorization~\eqref{eq:triangle-monodromy}, which has $b_2(X_1)=3$, matching $b_2$ of the minimal resolution. Next is the filling $X_2$, given by the allowable Lefschez fibration whose monodromy factorization we derive from~\eqref{eq:triangle-monodromy} by substituting the factor $t_{\delta_1} t_{\delta_2} \cdots t_{\delta_{3+k}}$ with 12 Dehn twists along homologically essential curves, coming from the monodromy of an elliptic pencil with $9$ base points \cite{BMVHM}. (For $k > 6$, there is no such pencil.) Here $b_2(X_2)= 12$, which agrees with $b_2$ of the Milnor fiber(s). Lastly, we obtain the filling $X_3$ by applying a lantern substitution to~\eqref{eq:triangle-monodromy} as before: we replace the Dehn twists around the curves $\{a_1, a_3, \delta_1, \delta_2\}$ by three other positive Dehn twists. So $b_2(X_3)=2$, and this filling is unexpected.

\begin{remark} There are smoothable triangle singularities where unexpected fillings detected by the Euler characteristic as above do not exist: for the contact link $Y=\Sigma(2, 3, 7)$ of the hypersurface singularity $\{x^2+y^3+z^7=0\} \subset \mathbb{C}^3$ (a triangle with $2,3,7$ legs),  every negative definite Stein filling  has the same intersection form as the minimal resolution \cite[Theorem~1.19]{EG}.
  \end{remark}

 \subsection{Non-smoothable cusp singularities} Rational surface singularities are always smoothable, and the Milnor fiber for one of the smoothings is diffeomorphic to the minimal resolution. In the minimally elliptic case, a new phenomenon appears: 
 some of the singularities in this class are  {\em non-smoothable}.

 The following dual resolution graphs correspond to cusp singularities. The graph is a cycle with $r$ vertices corresponding to exceptional curves of genus 0, and the self-intersections of the
 exceptional curves are the integers $-a_i$, where
 $a_i\geq 2$ for all  $i=1, \dots, r$, and  $a_j\geq 3$  for some  $j$.
 See Figure~\ref{cusp-graph}. 
 \begin{figure}[htb]
	\centering
	\includegraphics[scale=0.8]{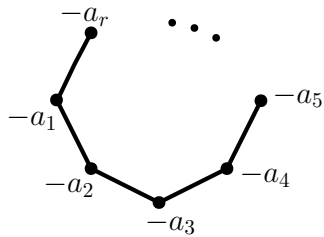}
	\caption{The minimal resolution graph for a cusp singularity. }
	\label{cusp-graph}
\end{figure}

By
 \cite{Wahl-smooth}, a singularity in this family is not smoothable
 whenever $m>r+9$, where the multiplicity $m$ is given by $\sum_{i=1}^r (a_i-2)$. We consider non-smoothable cusp singularities with
 \begin{equation}
  \sum_{i=1}^r (a_i-3) >9.
 \end{equation}

  The non-smoothable cusp singularities are taut, that is, there is only one analytic structure compatible with the given resolution graph. For this analytic structure, there are no smoothings and thus no Milnor fibers. The only Milnor filling in this case is given by the minimal resolution. We can construct unexpected fillings for many of these links:
 
 \begin{theorem} \label{thm:cusps} For any given $N>0$, there exists a non-smoothable cusp singularity as above whose link admits at least $N$ Stein fillings with different $b_2$.
 \end{theorem}

\begin{proof} A compatible open book for the contact link is constructed
from the resolution graph of Figure~\ref{cusp-graph} using \cite[Theorem 1.1]{GayMark}. Indeed, \cite{GayMark} describes a compatible Lefschetz fibration for the symplectic plumbing given by the minimal resolution. The fiber of the Lefschetz fibration 
is a torus, with subsurfaces given by punctured spheres that correspond to the exceptional curves. The monodromy is given by the product of positive Dehn twists around the necks connecting these subsurfaces, together with positive boundary twists around each puncture. The total number of vanishing cycles on the boundary of the subsurface corresponding to the $i$-th exceptional curve equals $a_i$. If $a_i=2$, then the corresponding subsurface is an annulus with no additional punctures (resulting in two parallel vanishing cycles). See Figure~\ref{cusp-obook}.
\begin{figure}[h]\label{cusp-obook}
\centering
\includegraphics[scale=1]{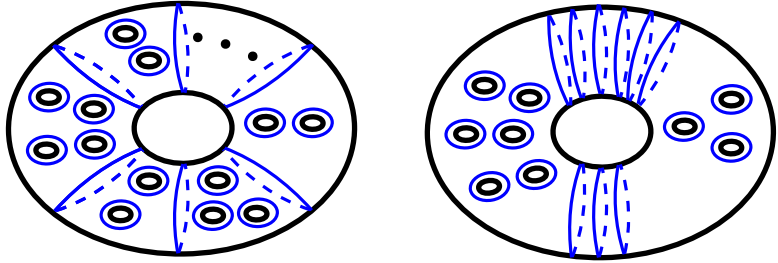}
\caption{An open book for the link of a cusp singularity has genus one; each exceptional curve of self-intersection $-a_i$ corresponds to a cylindrical subsurface of the page with $a_i-2$ punctures. When $a_i=2$, the subsurface is an annulus with no punctures. The monodromy is the product of the positive Dehn twists around the punctures and the meridianal curves separating the subsurfaces.
Left: an open book for the cusp with $a_1=4, a_2=6, a_3=4, a_4=5, a_5=4, \dots$.
Right: an open book for the cusp with $a_1=8, a_2=a_3=2, a_4=5, a_5=a_6=\dots=a_{10}=2$.
}
\end{figure}

Under favorable numerical conditions on $\{a_i\}$, the monodromy admits different positive factorizations, which allows to generate different Stein fillings.
If $a_i=4$, then the page contains a subsurface given by a sphere with 4 punctures, with boundary Dehn twists around each puncture; we can apply the lantern substitution to the product of these four Dehn twists. Suppose that $a_1=a_3=\dots=a_{2N+1}=4$, for $r>2N+1$. Then at least $N$ lantern substitutions can be performed on pairwise disjoint subsurfaces; different number of such substitutions yields Stein fillings with different $b_2$.

Additional lantern substitutions are often possible, and under other numerical hypotheses, one may be able to use generalized lantern relations; e.g. the monodromy factorization illustrated on the right of Figure~\ref{cusp-obook} allows for daisy substitutions  \cite{BMVHM}.
\end{proof}

\begin{remark} The lantern substitution in the monodromy factorization of a Lefschetz fibration is equivalent to symplectic rational blowdown of the $(-4)$ rational curve, \cite{EndoGurtas}: we replace the subgraph corresponding to the resolution of the cone on the $(-4)$ curve by its Milnor fiber with $b_2=0$.
The rational blowdown is seen in the resolution graph; our Lefschetz fibration argument demonstrates that we get a Stein (rather than just symplectic) filling as a result. Certain other monodromy substitutions, such as the daisy substitution, are equivalent  to blowdown along more complicated plumbings, \cite{HMVHM}. Moreover, 
the blowdown and the monodromy substitutions are equivalent as symplectic operations, \cite{GayMark}. It is quite plausible that  the fillings obtained via generalized rational blowdown 
can be obtained from P-resolutions of the corresponding singularity, but we will not pursue this direction in the present paper.   
\end{remark}

\section{Milnor fibers of sandwiched singularities}\label{s:sa}

The main result of this section is the following:

\begin{theorem}\label{th:finite}
Let $(Y,\xi)$ be the link of a sandwiched singularity equipped with its canonical contact structure.
Then $(Y,\xi)$ admits only finitely many Milnor fillings. More precisely, the set of Milnor fibers associated with all the smoothings
of all analytic realizations of the topological type identified by $Y$ is finite.
\end{theorem}

The proof has several steps. To describe smoothings and the corresponding Milnor fibers, 
we rely on  de Jong--van Straten's construction~\cite{dJvS}, which allows to reduce the question to certain deformations of curves.

\subsection{Smoothings of sandwiched singularities.} We closely follow \cite{dJvS}.
Let us fix a possible  topological type of a sandwiched singularity. It is characterized by its minimal resolution graph $\Gamma$, which is
a ``sandwiched graph''.
This means that $\Gamma$ is connected and it can be completed by several $(-1)$ vertices $\{F_i\}_{i=1}^r$  (and the supporting edges)
such that the completed graph is smooth: it can be blown down to the empty graph.  Let us
regard the graph  as the dual graph of an exceptional curve, and add to each $F_i$ a generic curvetta $\widetilde{C}_i$, which is a small complex disc transverse to $F_i$.  In the language of graphs, each curvetta is represented by an arrowhead supported by the $F_i$'s, in this way we get a graph $\widetilde{\Gamma}$. (For the correspondence between dual graphs and exceptional curves of resolutions see e.g. \cite{Nem,Nem-book}.)
Then, in this geometric presentation,  if we  blow down all the exceptional curves,
then  the image of $\cup_{i=1}^r \widetilde{C}_i$ will be   a plane curve singularity
$(C,0)=\cup_i (C_i,0)$, with $r$ local irreducible components,  such that $\widetilde{\Gamma}$ is its  embedded  resolution graph. 
There can be several combinatorial ways to complete the graph $\Gamma$ by the $(-1)$ vertices as above. We fix one such combinatorial choice,  
this will provide the graph  $\widetilde{\Gamma}$. Then, 
the equisingularity type of $(C,0)$ (that is, its embedded topological type)
 is independent of all the other choices; it depends only on~$\widetilde{\Gamma}$.

One can proceed in the opposite order as well. We start with a decorated isolated curve singularity $(C,l)$.
Here $(C,0)$ is an  isolated  plane curve singularity,  and each component $(C_i,0)$ of $(C,0)$ is decorated by an integer
$l_i$,  $i=1, \dots, r$.
 We consider the minimal embedded resolution graph
$\widetilde{\Gamma}_{min}$ of $(C,0)$
and then, for every $i$,  we blow up several times
 the intersection point of the strict transform of $(C_i,0)$
 with the exceptional curve
 (the number of blow ups is determined from $l_i$ and $\widetilde{\Gamma}_{min}$, for details see \cite{dJvS}).
In this way we obtain $\widetilde{\Gamma}=\widetilde{\Gamma}(l)$.
Then the non $(-1)$ vertices form the resolution graph $\Gamma$ of the constructed sandwiched singularity.  In particular, the equisingularity type of
$(C,0)$ and $l$  determines the topological type of the sandwiched singularity $(X,0)$, and the analytic type of $(C,0)$ and $l$ determines the
analytic type of $(X,0)$.

This shows that if we consider all the decorated analytic plane curve germs  with fixed equisingularity type, as determined by
$\widetilde{\Gamma}(l)$, then the de Jong--van Straten  construction provides all the analytic sandwiched singularities with topological type fixed by their
common resolution graph $\Gamma$.

Next, we connect smoothings of $(X,0)$ with deformations of $(C,l)$: picture deformations of $(C,l)$ provide smoothing components of $(X,0)$.
Let us fix a decorated germ $(C,l)$.
A  one-parameter  deformation of $(C,l)$  consists  of  a one-parameter $\delta$--constant deformation $\{C_t\}_t$ of $(C,0)$
and a flat deformation $\{l_t\}_t$ of the scheme associated with $l$. This deformation is a `picture deformation'  if for $t\not=0$ the divisor $l_t$ is reduced.
In such a case  $C_t$ for $t\not=0$ can have   only ordinary $m$--tuples  as singularities. A part of the (reduced) points of $l_t$
must sit in these  $m$--tuples $(m\geq 2$), and there might be some other free reduced points supported on smooth points of  $C_t$.
Such a picture deformation generates a smoothing of $(X,0)$, hence a Milnor fiber.

Let us recall an important property of a deformation
$(\Sigma,0)  \stackrel{\pi}{\rightarrow} (T,0)$  of $(C,0)$ over a smooth 1--dimensional germ $(T,0)$. Let $\widetilde{C}$  be the normalization of $(C,0)$
(as a smooth multigerm, which can be identified with $r$ small discs).
By a characterization of Teissier  \cite{T80} (see also
\cite{CL06}), a deformation $\pi$
is $\delta$--constant if and only if the normalization
morphism of $\Sigma$ has the form $\tilde{C}\times T\to \Sigma$, and it is
the \emph{simultaneous normalization} of the fibers of $\pi$.
In particular, if one has a $\delta$--constant embedded deformation of a germ of
\emph{plane} curve $(C,0)$ such that the general fiber $C_t$ has only
ordinary multiple points, then after a choice of a Milnor
representative of the morphism, the general fibers are
\emph{immersed discs} in the $4$--ball (in $\mathbb{C}^2$); their number is the number 
of local irreducible components of $(C,0)$. (See e.g. \cite[2.2]{NPP}.)

\subsection{Fixing the topological type and combinatorial data} \label{subsect:smoothings}
We consider the family of  all the decorated analytic plane curve germs  with fixed equisingularity type and constant decorations. Since the
 picture deformations of any fixed member of the family
provide all the smoothings of the corresponding sandwiched singularity,
we will obtain all the Milnor fibers   of all  analytic structures (supported by a fixed topological type) when we consider all the deformations of all the members.

If we consider a deformation of $(C,0)$ and 
 fix a small Milnor ball $B^\epsilon \subset \CC^2$ centered at the origin, then for $|t|$ small, 
the singularities of $C_t $ and the support of
$l_t$  sit in $B^\epsilon$, and $C_t$ intersects $\partial B^\epsilon$ transversely. For $t\not=0$
they provide a topological picture: the union of immersed  discs $C_{t,i}$ in $B^\epsilon$, which intersect each other and themselves 
 along certain $m$--tuples. From this embedded data $(C_{t,i}, l_{t,i})_i \subset B^\epsilon$ one constructs the Milnor fiber of the  smoothing:
 
 \begin{prop} \label{prop:djvs} \cite[Prop. 5.1]{dJvS} The Milnor fiber associated to a picture 
 deformation $(C_t, l_t)$ is obtained by blowing up $B^\epsilon$ at all the points $l_{t, i}$ and taking the complement 
  of the union of the strict transforms of $C_{t, i}$.
 \end{prop}
Since we are interested in fillings of the link $(Y, \xi)$, we usually consider the compact version of the Milnor fiber, which is the complement of small tubular neighborhoods of the strict transforms in a closed Milnor ball. (Technically, this construction produces a manifold with corners; we will assume that the corners are smoothed.)

Combinatorially, the intersections of the immersed discs $C_{t,i}$  and the position of the points $l_{t,i}$ are encoded by an incidence matrix of size
$r \times \#\{$points$\}$ with non-negative entries
$v_{ij}$, where $v_{ij} $ is  the multiplicity of the  branch $C_{t,i}$ at the point $p_j$. These  entries satisfy several identities  (cf. page 484  of \cite{dJvS}), e.g., for all $i$:
\begin{equation}\label{eq:1}
\sum_j v_{ij}(v_{ij}-1)=2\delta(C_i,0), \ \ \mbox{or} \ \ \ \sum_j v_{ij}=l_i.
\end{equation}

Denote the set of possible decorated embedded topological types   $(C_{t,i}, l_{t,i})_i \subset B^\epsilon$
by  $\E$, the set of incidence matrices by $\I$ and
the set of Milnor fibers  by $\F$  (all realized by analytic deformations).
Then the construction \cite[Prop. 5.1]{dJvS} provides a map $\beta : \E\to \F$, and  $\alpha:\E\to \I$ associates the combinatorial incidence matrix $(v_{ij})_{ij}$ to a decorated embedded topological type  $(C_{t,i}, l_{t,i})_i \subset B^\epsilon$ (i.e. to a
picture deformation), cf. \cite[4.3]{dJvS}.

Let us make the following observation regarding the maps $\alpha$ and $\beta $ above.
Since 
the graph $\widetilde{\Gamma}(l)$ is fixed, 
hence the embedded topological type of $(C,0)$ is fixed,
its delta invariant $\delta$, and Milnor number $\mu$ are fixed as well. Furthermore, each $l_i$ is also fixed.  This shows that
the set $\I$ is finite. Indeed, we can define the set of possible combinatorial incidence matrices $\I_{comb}$ satisfying (say) the identities as in
(\ref{eq:1}), then $\I_{comb}$ is finite by  (\ref{eq:1}),
and  the analytically realized matrices $\I$ form a subset of $\I_{comb}$.

What is not immediate is that the cardinality of $\alpha^{-1}(I)$ ($I\in\I$)  is finite. In principle, we can imagine many different decorated  immersed discs
with the very same combinatorial incidence matrix. Even more, starting from a concrete immersion, one can even try to construct by surgeries
many other ones  with unmodified incidence matrix and unmodified boundary of the filling.

On the other hand, it is worth pointing out that
the main result of  \cite{NPP}  implies  that for $I_1, \, I_2\in\I$, $I_1\not=I_2$, $\beta(\alpha^{-1}(I_1))\cap \beta(\alpha^{-1}(I_2))=\emptyset$.
Hence,  Milnor fibers corresponding to  different incidence matrices are different  (up to diffeomorphism  fixing certain boundary data). In particular, $\#\F \geq \#\I$. We wish to complete this lower bound by proving that $\F$,  in fact, is finite.

\subsection{A parameter space of curves} \label{subsect:space-of-curves}
We would like to consider possible analytic structures of decorated curve germs $(C,l)$ with fixed topology $\widetilde{\Gamma}(l)$.
To construct a parameter space,  it is helpful  to globalize the local curves  into projective ones. It is known that for any plane curve singularity $(C,0)$, there exists
$d_0$ sufficiently large (depending on the topological type of  $(C,0)$) such that for any
$d\geq d_0$  there exists a projective plane 
curve of degree $d$ with only one singular point analytically equivalent to $(C,0)$, see e.g.
\cite[Corollary 4.5.15]{GLS:curves} or \cite[subsection 1.1.3.5]{GLS:curves}.
In fact, even the irreducibility of the projective curve is guaranteed
\cite[Theorem 4.6.4]{GLS:curves}.

Next note that if $(C,0)$ has Milnor number $\mu$ (determined by its embedded topological type) then the equation of $(C,0)$ is
$(\mu+1)$--determined. This means that for any $k\geq \mu+1$, the analytic type is determined by the $k$--jet of the equation, see e.g. page 14 of \cite{Arnold}.
Furthermore, the universal deformation space of $(f=0,0) =(C,0)$  has finite dimension $\tau\leq \mu$, and it has the form
   $f+\sum_{j=1}^\tau t_jg_j $, where $t_j\in(\CC,0)$ and
   $\{g_j\}_j$ induce a basis of the Tjurina algebra $T(f):=\CC\{x,y\}/(f, \partial f/\partial x, \partial f/\partial y)$ of dimension $\tau$.
   Note that if we add to  $g_j$ a monomial of degree larger than $\mu$, then its class in $T(f)$ will not change.
   This shows that for any analytic type $(C,0)$ one can find $ \tau$ polynomials $g_j$, all with degrees bounded by  $\mu$, such that
  $f+\sum_{j} t_jg_j $ is a versal deformation of $f=0$, hence it ``contains''  all the deformations of $(C,0)$.

Let us consider now the space $\C =\CC\PPP^M$ of all projective curves in $\CC\PPP^2$ of degree $d$ for  $d$ is {\it sufficiently large} ($M=d(d+3)/2$). By the above discussion
any local analytic singularity  with graph $\widetilde{\Gamma}(l)$ appears as singularity of a projective curve $\bar{C}\in \C$, and also, the neighborhood of
$\bar{C}$ in $\C$ serves as  a versal deformation of $\bar{C}$. In this way, all the local analytic types  of plane singular germs with topology
identified by  $\widetilde{\Gamma}(l)$ and all their deformations can be realized inside $\C$.

Let $\C(\Gamma)$ be the set of all curves $\bar{C}\in \C$ such that $\bar{C}$ has exactly one singular point  of embedded topological type
$\widetilde{\Gamma}(l)$. Then $\C(\Gamma)$ is  a non-empty  quasiprojective variety in $\C$ (see
e.g. pages 59--61 in \cite{GLS:curves}).
By \cite{WahlES}, for any fixed $(C,0)\subset (\CC^2,0)$, the strata $ES$  parametrising the equisingular deformations  of $(C,0)$ inside the space of
semiuniversal deformation space  is a smooth subvariety. Since our parameter space is versal for any $(C,0)$, the same property holds: $\C(\Gamma)$
is a smooth quasiprojective  subvariety of $\C$ (see also page 61 of \cite{GLS:curves}). Moreover, by \cite[Theorem 4.6.4]{GLS:curves} $\C(\Gamma)$ is
even irreducible (a fact which will not  be used here).

Consider the universal family $\UU= \{(\bar{C},p)\in\C\times \CC\PPP^2\,:\, p\in \bar{C}\}$,
 the first  projection  $q:\C\times \CC\PPP^2\to \C$
   given by $(\bar{C},p)\mapsto \bar{C}$,  and
  its restriction to $\UU$, $q_\UU:\UU\to\C$.
   Set $\UU(\Gamma):= q_\UU^{-1}(\C(\Gamma))\subset \UU$.
Then the restriction $q_\UU:\UU(\Gamma)\to \C(\Gamma)$ admits an algebraic section,
 $\bar{C}\mapsto (\bar{C},p_{\bar{C}})$, where $p_{\bar{C}}$ is the unique singular point of $\bar{C}$.  Let us denote by $\PP(\Gamma)$ this  image $(\bar{C}, p_{\bar{C}})_{\bar{C}\in \C(\Gamma)}$ in $\UU(\Gamma)$.  Then
 $\PP(\Gamma)\to \C(\Gamma)$ is an isomorphism and $\PP(\Gamma)$ is a smooth subvariety of $\C\times \CC\PPP^2$.

 \subsection{Uniform Milnor radius} Now, a decorated curve germ $(C,l)$ with fixed topology $\widetilde{\Gamma}(l)$ corresponds to a point $C \in \C(\Gamma)$ with appropriate decoration. 
 Given a choice of picture deformation  $(C_t, l_t)$ of $(C, l)$, the Milnor fiber of the corresponding smoothing is obtained from a Milnor ball $B^\epsilon$ of $C$ as in Proposition~\ref{prop:djvs}. The radius $\epsilon$ of the Milnor ball a priori depends on $C$. 
 We will need to be able to  choose a {\em uniform} radius in a neighborhood of $C \subset \C(\Gamma)$; then a uniform positive radius can be found over any subset of $\C(\Gamma)$ with compact closure.  Existence of uniform radius is  guaranteed by the following 
 local argument (it is ``classical'', but far from trivial).
 
 \begin{lemma} \label{lem:radius} Fix  $C \in \C(\Gamma) \subset \CC\PPP^M$. Then there is a (Euclidean) neighborhood $V \subset \C(\Gamma)$ of $C$ and $\epsilon_0>0$ such that for every $\bar{C} \in V$ and $\epsilon \leq \epsilon_0$, the curve $\bar{C}$ intersects $\d B^\epsilon(p_{\bar{C}})$ transversely.       
 \end{lemma}

 \begin{proof} First, we note that for a deformation of an isolated plane curve singularity  with  constant embedded topological type,
  or with  constant Milnor number,   the multiplicity is also  constant. 
This happens  because the topological type determines the 
   Puiseux pairs of the components, hence the multiplicity as well. 
Then, Teissier's $\mu^*$--sequence (in this curve case this consists of the  Milnor number and the multiplicity)  is constant, which shows that 
locally over small balls of $\PP(\Gamma)$, 
the total space $\UU$ of the family is Whitney equisingular along the  stratum $\PP(\Gamma)$ of singular points \cite{TeissierW}. On the other hand,  
 non-constant Milnor radius violates Whitney condition (b): take as a sequence of secants the origin joined with the tangency point with the sphere. Indeed, assume that a uniform Milnor radius does not  exists. This means that there exists a sequence $\{\bar{C}_n\}_n $ in $\C(\Gamma)$ with limit
$\lim_n\bar{C}_n=\bar{C}_\infty \in\C(\Gamma)$,  and a sequence $\{\epsilon_n\}_n$, $\epsilon_n>0$, $\lim_n \epsilon_n=0$, such that 
$q^{-1}(\bar{C}_n)$ is tangent to $\partial B_{\epsilon_n}$ at a point $p_n $  with $\lim_n p_n=p_{\bar{C}_\infty}$. Then one can take the secant
lines $\overline{p_{\bar{C}_n}p_n}$. Since the tangent space $T_{p_n}q^{-1}(\bar{C}_n)$ is contained in $T_{p_n}\partial B_{\epsilon_n}$, 
we have the ortogonality $\overline{p_{\bar{C}_n}p_n} \perp T_{p_n}q^{-1}(\bar{C}_n)$. In particular,
the inclusion $\lim_n \overline{p_{\bar{C}_n}p_n}\subset \lim_n  T_{p_n}q^{-1}(\bar{C}_n)$, formulated by Whitney condition (b),  cannot hold.

 Hence, constancy of the $\mu^*$--sequence implies the existence of the uniform Milnor radius.  For more details see also  page 261 and Theorem 5.3.1 
 in \cite{BG} and the corresponding references therein. 
 
 It may be helpful to note that the lemma admits an elementary proof if all the irreducible components of the germ $C$ are smooth (in the de Jong--van Straten setting, such decorated germs correspond to rational surface singularities with reduced fundamental cycle). In this case, existence of uniform Milnor radius follows simply from smoothness of individual branches and transversality.  
 \end{proof}

In the subsequent steps, we will consider picture deformations of the curves $C$ inside these Milnor balls. Different combinations of $m$-tuple singularities of 
picture deformations give a stratification of the corresponding space of curves. We will show 
that the corresponding Milnor fibers look the same within a component of a fixed stratum 
(more precisely, they form a $C^{\infty}$ locally trivial fibration). Finiteness of the set of Milnor fillings will then follow from finiteness of the related stratification.

\subsection{Curves with $m$-tuple singularities: stratification}
Consider
$$\cS=\{\bar{C}\in\C\,:\, \bar{C} \ \mbox{has only $m$-tuple type singularities}\}.$$
The space $\cS$ has an interesting stratification, it
decomposes into several disjoint strata  $\cS_{\overline{m}}$ according to the type
 $\overline{m}=(m_1, m_2, \dots)$ of the singular points.
In fact, we consider only those strata  of $\cS$ which appear in the neighborhood of $\C(\Gamma)$, that is, their closure intersect $\C(\Gamma)$ nontrivially.
Furthermore, we assume for any such component that they  represent  $\delta$--constant deformations of certain points of $\C(\Gamma)$.  In other words,
there exists a real analytic path $\gamma:[0, \tau]\to\C$, $\gamma(0,\tau]\in \cS_{\overline{m}}$, $\gamma(0)\in \C(\Gamma)$ such that $\gamma(t)$ is a $\delta $--constant deformation of $\gamma(0)$, with $m$--tuple singularities of type $(m_1,m_2,\ldots)$, so that  the type is constant for $t\not=0$.
This shows that the constant delta value of the strata $\C(\Gamma)$ bounds the possible collection
$(m_1, m_2, \ldots)$, since $\sum_i m_i(m_i-1)/2=\delta$.

Since all the $m_i$ values are bounded, if we choose $d$ (the degree of the curves from $\C$)
sufficiently large then each $\cS_{\overline{m}}$  is quasiprojective, and in fact, it is smooth, cf. \cite[Ch. 4]{GLS:curves}. Moreover, the number of irreducible, or connected components of any  $\cS_{\overline{m}}$  is finite.

If $(C,0)$ itself is an $m$--tuple, then $\C(\Gamma)$ is such a stratum $\cS_{\overline{m}}$, with only one entry $m_i$ which is $m=r$. 

In general, the varieties
$\{\cS_{\overline{m}}\}_{\overline{m}}$ form an interesting stratification, e.g.
$\cS_{(3)}$ sits in the closure of $\cS_{(2,2,2)}$. But we emphasize again: $m$--tuple points parametrized by a fixed component of  $\cS_{\overline{m}}$
do not merge, do not split, and their number is constant. 

It can happen that the closure of a certain $\cS_{\overline{m}}$ contains all $\C(\Gamma)$. This means that any analytic germ of $\C(\Gamma)$ admits a $\delta$--constant deformation of that type $\overline{m}$. But it can also happen that the closure of a certain $\cS_{\overline{m}}$ contains only a proper subspace of  $\C(\Gamma)$. This means that only certain restricted  analytic germs of $\C(\Gamma)$ admit  $\delta$--constant deformations of that type $\overline{m}$. See page 501 in \cite{dJvS} for a concrete example. 

 \subsection{Local triviality over strata} \label{subsect:loctriv} We focus on a stratum   $\cS_{\overline{m}}$ for a fixed 
 choice of $\overline{m}$. We will consider two cases, as the setup is slightly different: (1) $\cS_{\overline{m}} = \C(\Gamma)$ and (2) 
 $\cS_{\overline{m}}\not= \C(\Gamma)$. In the first case,  $(C,0)$ must be an $m$--tuple, and
$\overline{m}=(m)$, where  $m=r$. In the second case, $\cS_{\overline{m}}\cap \C(\Gamma)=\emptyset$, but  the intersection of the closure of $\cS_{\overline{m}}$ with $\C(\Gamma)$ must be non-empty if $\cS_{\overline{m}}$ contains any picture deformations of a decorated curve from $\C(\Gamma)$.

Take $C \in \C(\Gamma)$, and assume for convenience that $p_{C}=0$. Let $\epsilon>0$ and let a neighborhood
$V\subset \C(\Gamma)$ of  $C$ be as in Lemma~\ref{lem:radius}, so that $B^\epsilon= B^\epsilon(0)$ is a Milnor ball for every curve in $V$. We fix this $B^\epsilon \subset \CC^2 \subset \CC\PPP^2$. 

In Case (1), we consider the projection 
\begin{equation} \label{q-proj1}
(V \times B^\epsilon,  (V  \times B^\epsilon) \cap \UU)  \stackrel{q}{\to} V. 
\end{equation}

In Case (2), we can find a neighborhood $W= W_{C}$ of $C$ in the space $\C$ such that for every $\bar{C} \in W$, all singular points of  $\bar{C}$ are contained in the interior of $B^\epsilon$, and the intersection of $\bar{C}$ and $\d B^\epsilon$ is transverse. (Here and below, $B^\epsilon$ is  the {\em closed} ball.)
We then consider the projection
\begin{equation} \label{q-proj}
((W \cap \cS_{\overline{m}})  \times B^\epsilon,  ((W \cap \cS_{\overline{m}})  \times B^\epsilon) \cap \UU)  \stackrel{q}{\to} W \cap \cS_{\overline{m}}.  
\end{equation}

In each case, the fiber over the curve  $C_t$ representing a point in $V$ resp. $W\cap \cS_{\overline{m}}$  is $(B^\epsilon, C_t \cap B^\epsilon)$.  We would like to show that the fibration \eqref{q-proj1} resp.  \eqref{q-proj}
is locally trivial in the appropriate sense. More precisely, for a picture deformation $C_t$ endowed with a decoration, we need to understand the fibration obtained by the procedure of constructing smoothings and Milnor fibers of the corresponding surface singularities, as in Proposition~\ref{prop:djvs}. Case (1) is the easier of the two, and we will comment on it after treating Case (2) 
in detail.  We now assume that $\cS_{\overline{m}}\not= \C(\Gamma)$, and $W$ is 
chosen as above.

By our assumptions,  every $C_t \in W \cap \cS_{\overline{m}}$ has the same  $\delta$ as $(C,0)$, all its singular points are $m$--tuples,
these  singular points are in  $B^\epsilon$
 and $C_t$  intersects $\d B^\epsilon$  transversely.

Then $C_t$ can be completed to a picture deformation $(C_t,l_t)$. The points of $l_t$ supported
at the singular points are determined by $C_t$, the other (free)
points,  still sitting in the interior of  $B^\epsilon$,
 are distributed on the regular part
of the components $C_{i,t}$. The number of free points on each $C_{i,t}$, say
$l_{i,t}'$, is fixed from the combinatorics. Indeed, if $p_j$ are the singular points of $C_t$ and
$v_{ji}'$ is the multiplicity of $C_{t,i}$ at $p_j$, then
$l_{i,t}'=l_i-\sum_j v_{ij}'$.

Each completion 
 provides a picture deformation, hence an element of $\E$  (decorated immersed discs in $B^\epsilon$). By Proposition~\ref{prop:djvs},
if we blow up $B^\epsilon$ at the points of the  support  of $l_t$ and delete the strict transform of the curves $\{C_{i,t}\}_i$ we get the Milnor fiber of 
the corresponding picture deformation (or of the corresponding smoothing).

Regarding these blow ups, we proceed in two steps. First, we blow up the singular points of $C_t$. As a byproduct of this step, we will also prove the appropriate 
local triviality of the map  $q$  from \eqref{q-proj}. Then, as a second step,  we blow up the free points as well. Denote the set of singular points of $((W \cap \cS_{\overline{m}})  \times B^\epsilon) \cap \UU$ by $\rm{Sing}_{\overline{m}}$.

In order to prove the local triviality of  $q$  from 
\eqref{q-proj}, for any point $\bar{C}_0$ of the base $B_{\overline{m}}$ we have to find a 
 small  ball  (or a contractible open set) $W' \subset W \cap \cS_{\overline{m}}$ in the base space which contains $\bar{C}_0$, such that  $q$ is trivial over $W'$.
 We choose $W'$ as follows. 

  By construction,  the number of singular points in any fiber $q^{-1}(\bar{C})\cap B^\epsilon$  is the same, say $k$. 
In particular, the restriction $ q: {\rm Sing}_{\overline{m}}\to W \cap \cS_{\overline{m}}$ is a regular $k$--covering.  We choose $W'$ contractible and 
so small that  this regular (analytic) covering over $W'$ is trivial.

We wish to show the triviality of 
\begin{equation} \label{eq:triv-Wprime}
(W'  \times B^\epsilon,  (W' \times B^\epsilon) \cap \UU)  \stackrel{q}{\to} W'.  
\end{equation}

 From the choice of $W'$, the covering
 $(W' \times B^\epsilon) \cap {\rm Sing}_{\overline{m}}$  over $W'$ is trivial, hence it has $k$ 
regular sections 
$s_i: W' \to (W' \times B^\epsilon) \cap {\rm Sing}_{\overline{m}}$, 
$1\leq i\leq k$, so that $q\circ s_i={\rm id}_{W'}$.

Since $W'$ is smooth, the image 
$s_i(W') $ is also smooth, and $s_i$ and $q$ realize   complex analytic isomorphisms between $W'$ and $s_i(W')$.
Now,  $\cup_i s_i(W')$ is a smooth complex analytic subspace of complex codimension two in a smooth complex analytic space 
$W'  \times B^\epsilon$. We blow up this subspace and 
denote the total space by $\BB(W')$, the blow up projection by $b$,  
$q\circ b=\widetilde{q}$, the exceptional divisor by $\widetilde{E}$, the strict transform of $(W' \times B^\epsilon) \cap \UU$ by 
$\widetilde{St}(W')$. Since all singular points of the curves from $W'$ are $m$-tuples, the strict transform $\widetilde{St}(W')$ is smooth, and it intersects the smooth subspace $\widetilde{E}$ transversely. 

Over a point $p\in W'$ representing a curve with $m$-tuple singularities, $\widetilde{q}^{-1}(p)$ is the blow up of $B^\epsilon$ at the singular points of the curve $p$,  
$\widetilde{q}^{-1}(p)\cap \widetilde{E}$
consists of $k$ copies of  $\CC{\mathbb P}^1$ which are the 
exceptional curves of this blowup of $B^\epsilon$, 
$\widetilde{q}^{-1}(p)\cap \widetilde{St}(W')$ is the strict transform of the curve $p \cap B^\epsilon$,  and 
$\widetilde{q}^{-1}(p)\cap \widetilde{St}(W')\cap \widetilde{E}$ consists of $\sum_jm_j$ points corresponding to the type of the $m$--tuples.
Furthermore, $\widetilde{St}(W')$ intersects $
\d\BB(W')$ transversely as well.

Then, using the transversality and an Ehresmann-type argument (see below), we obtain  that 
\begin{equation}\label{loctriv-W}
(\BB(W'), \partial \BB(W'), \widetilde{E}, \partial \BB(W')\cap \widetilde{St}(W'), \widetilde{E}\cap \widetilde{St}(W'))\to W'
\end{equation}
is a $C^\infty$ locally trivial fibration of 5--tuples of spaces over $W'$. But  $W'$ being contractible, this is in fact a trivial smooth  fibration. 

By an Ehresmann-type argument, we mean the following. The classical Ehresmann theorem asserts that a proper submersion between two closed smooth manifolds is a $C^\infty$ locally trivial fibration \cite[Section 8.5]{Dundas}. There are versions of this result for manifolds with boundary (the restriction of the map to the boundary is required to be a submersion), as well as for the relative case (see for example \cite[Theorem 1.2.16]{Lipschitz-Geom}, \cite{MOflow}). 
Because the construction involves the projection $\tilde{q}$ which is tautologically a submersion, and transversality ensures that all intersections are smooth submanifolds, it is not hard to see that the Ehresmann hypotheses are satisfied for the  5-tuple of \eqref{loctriv-W}, as well in our subsequent applications of a similar argument. We leave the details to the reader. 

A similar argument works in Case (1), but since each curve $\bar{C} \in V$ is an $m$-tuple, there is only one singular point to blow up in each fiber. We obtain a $C^\infty$ locally trivial fibration of 5--tuples of spaces over $V$,
\begin{equation}\label{loctriv-V}
(\BB(V), \partial \BB(V), \widetilde{E}, \partial \BB(V)\cap \widetilde{St}(V), \widetilde{E}\cap \widetilde{St}(V))\to V. 
\end{equation}

Blowing down the locally trivial fibrations \eqref{loctriv-V} and \eqref{loctriv-W}, we
can get local triviality statements for \eqref{q-proj} and \eqref{q-proj1}. 
 The local triviality for \eqref{q-proj1} means the following.  The first entry of the total space of the pair is smooth (and the fiber, the ball, as well), so in this case the usual definition works (of course, this is just a product fibration). 
The fiber of the second entry consists of $m=r$  embedded  discs with $m$--tuple singularities. 
Here, at the smooth points it behaves as a usual smooth locally trivial fibration, and at the singular points all the branches (automatically smooth) behave
as  smooth locally trivial  fibration,  and additionally, the tangent lines at the singular points remain distinct. (This follows from local triviality of  \eqref{loctriv-V} and properties of blow-down.) Local triviality of  \eqref{q-proj} holds in a similar sense: the fiber of the second entry consists of the union of immersed discs, we have the usual $C^\infty$ local triviality at smooth points, and at each of the $m$-tuple singularities each branch individually behaves as a smooth locally trivial  fibration, and the tangent lines to the branches are distinct. In particular, this implies that the fibrations \eqref{q-proj} and \eqref{q-proj1} are topologically trivial, and away from the singularities they are smoothly trivial. Note that they {\em cannot} be smoothly trivial everywhere, in general: it is well known that even in the case of a 4-tuple of lines, there is an additional  smooth invariant, the cross ratio.

\subsection{Back to  the decorations of the picture deformation.} To complete the construction of Milnor fibers
from a picture deformation $C_t$, we need to perform the additional blowups at the free points. We proceed as above, blowing up a family of free points in a fibration (we explain Case (2), and Case (1) is similar).

Let us fix $W \subset W \cap \cS_{\overline{m}}$  as above (and the fibration over it with all its properties already proved). Let $C_t$ and $C_{t'}$ be the curves corresponding to two points of $W'$.
Let us complete these curves  with decorations, that is with free points
$\{a_{t,j}\}_j\subset C_t$ and $\{a_{t',j}\}_j\subset C_{t'}$. (The other decoration points are the singular points, discussed in the previous subsection.)
The number of free points on $C_{i,t}$ and $C_{i,t'}$ agree: it is $l'_{i,t}$.

We claim that there exists smooth sections
$$s_{i,j}:W' \to (W' \times B^\epsilon) \cap \UU, \quad 1\leq j\leq l'_{i,t}$$
such that $s_{i,j}(t'')$ are free decoration points on the curve $q^{-1}(t'')\cap \UU$,
and over the points $t$ and $t'$ we have 
 $\{s_{i,j}(t)\}_j=\{a_{t,j}\}_j\subset C_t$ and $\{s_{i,j}(t')\}_j=\{a_{t',j}\}_j\subset C_{t'}$. 

This follows from the fact that for any $i$, the space $(C_{i,t}\cap B^\epsilon) \setminus {\rm Sing }[C_{i,t}]$  is a smooth connected manifold embedded smoothly in the smooth ambient space, and small tubular neighborhoods of these manifolds  do not intersect each other.
``Moving'' the free points on this smooth curve can be realized by a smooth flow which is identity in the complement of these tubular neighborhoods, and
flows nontrivially in the tubes  
(similarly to e.g. \cite[pages 22-23]{Milnor}).
Let $s_{i,j}(W')$ denote the smoothly embedded  image of $s_{i,j}$.

We wish to make the following comment. In the case of singular points of $C_t$, the sections $s_i$ ($1\leq i\leq k$) are complex analytic (or algebraic), 
but 
the above construction of the section of $s_{i,j}$ provides ``only'' a smooth embedding. In fact,  in the theory of de Jong and van Straten, the moving 
free points also appear as the support of an analytic/algebraic scheme, hence if we repeat their construction over the parameter 
space $\cS_{\overline {m}}$
we could even obtain that the sections  $s_{i,j}$ are complex analytic, or algebraic. However, we will avoid this technical part, and stay with the smooth
case which is still perfectly enough for smooth local triviality of the obtained fibrations and smooth  identification of the Milnor fiber.

Hence, in this case,  $W' \times B^\epsilon$ is a smooth complex analytic space and $\cup_{i,j} s_{i,j}(W')$ are  smooth 
$C^\infty$  (disjoint) subspaces of real  codimension four. 
We blow up  $W' \times B^\epsilon$ along  $\cup_is_i(W') $ (analytically) as in the previous subsection, and also along 
 $\cup_{i,j}s_{i,j}(W') $ smoothly. (Recall that the blow up of a complex surface $S$ can be performed at a point in the smooth category as well,  e.g. as 
  a smooth surgery, resulting in $S\#\overline{\CC\PPP^2}$. This can be done in smooth families as well.)
  Let us denote the total space of this additional blowup by $\BB_b(W')$ instead of the previous $\BB(W')$, and adopt all the previous  notation. 
  Then, by an identical argument, we get that 
  \begin{equation}\label{eq:final}
\widetilde{q}_b\,:\,   (\BB_b(W'), \partial \BB_b(W'), \widetilde{E}_b, \partial \BB_b(W')\cap 
\widetilde{St}_b(W'), \widetilde{E}_b\cap \widetilde{St}_b(W'))\to W' \end{equation}
is a $C^\infty$ (locally) trivial fibration of 5--tuples of spaces over $W'$.

\subsection{The Milnor fibers of picture deformations.} We are ready to prove finiteness of the set of Milnor fillings of the link of a sandwiched singularity $(X, 0)$. We fix a decorated germ $(C, l)$ that corresponds to $(X, 0)$ as in subsection~\ref{subsect:smoothings}, and consider the space of curves $\C(\Gamma)$ as in \ref{subsect:space-of-curves}. 
Fix a stratum $\cS_{\overline{m}} \neq \C(\Gamma)$ whose closure intersects $\C(\Gamma)$. (Again, we consider Case (2) of subsection \ref{subsect:loctriv}; Case (1) is similar but easier.) Suppose that $(C_t, l_t)$ is a picture deformation of a curve $C_0 \in \C(\Gamma)$, such that 
$C_t \in  \cS_{\overline{m}}$. If we fix $\epsilon>0$, the ball $B^\epsilon$ centered at the singular point $p_{C_0}$ of $C_0$,  and $W_0= W_{C_0}$ as in \ref{subsect:loctriv}, then 
$C_t \in W_0 \cap  \cS_{\overline{m}}$  for small $t$, so that all 
the singular points of $C_t$ lie in $B^\epsilon$, and $C_t$ is transverse to $\d B^\epsilon$. Applying Proposition~\ref{prop:djvs} 
to $(B^\epsilon, C_t \cap B^\epsilon)$ with the decoration $l_t$,
we get the Milnor fiber of the smoothing corresponding to $(C_t, l_t)$.
This Milnor fiber, associated to the picture deformation  $C_t$ 
 and the free decorations $\{a_{t,j}\}_j$, can  be constructed from the fibration $\widetilde{q}_b$ of \eqref{eq:final}, assuming that 
$C_t\in W'$:  we delete $\widetilde{St}_b(W') \cap (\widetilde{q}_b)^{-1}(C_t)$ (or its small tubular neighborhood) from  $(\widetilde{q}_b)^{-1}(C_t)$. The (local) triviality of \eqref{eq:final} implies that 
if we apply the same construction in the fiber $(\widetilde{q}_b)^{-1}(C_{t'})$ where $C_{t'}\in W'$ for a decorated curve $(C_{t'}, l_{t'})$, then the resulting 4-manifold with boundary $M_{(C_{t'}, l_{t'})}$ will be diffeomorphic to the Milnor fiber associated to  $(C_t, l_t)$. Here, we need to point out a technicality:  $M_{(C_{t'}, l_{t'})}$ is constructed by applying the procedure of Proposition~\ref{prop:djvs} 
to $(B^\epsilon, C_{t'} \cap B^\epsilon)$, where the ball $B^\epsilon = B^\epsilon(p_{C_0})$ depends on the curve $C_0$. Strictly speaking, Proposition~\ref{prop:djvs} says that if $(C_{t'}, l_{t'})$ is a picture deformation of 
another germ $(C_0', l')$, then we need to choose a Milnor ball $B^{\epsilon'}(p_{C'_0})$, and apply the procedure to this ball (assuming that $t'$ is small enough so that all singularities of $C_{t'}$ are contained in  $B^{\epsilon'}(p_{C'_0})$, and $C_{t'}$ is transverse to the boundary of this ball). However, since the same assumption is satisfied for $B^\epsilon(p_{C_0})$ and $B^{\epsilon'}(p_{C'_0})$, it is easy to see that the resulting 4-manifolds are diffeomorphic (use an interpolating family of balls satisfying the same condition). 
Therefore, $M_{(C_{t'}, l_{t'})}$ is indeed diffeomorphic to the Milnor fiber associated to the picture deformation $(C_{t'}, l_{t'})$  of  $(C_0', l')$. 

We can now finish the proof as follows. For each $\bar{C} \in \C(\Gamma)$, we choose an open set $W_{\bar{C}}$ as in subsection \ref{subsect:loctriv}, then $\mathcal{N=}
\cup_{\bar{C} \in \C(\Gamma)} W_{\bar{C}}$ is an open tubular neighborhood 
 of 
$\C(\Gamma)$.  For a fixed 1-parameter picture deformation $(C_t, l_t)$, all values of $t$ produce smoothings with diffeomorphic Milnor fibers, so we know that every Milnor fiber has a representative that comes from a curve in some connected component of $\cS_{\overline{m}}\cap\mathcal{N}$, for some stratum $\cS_{\overline{m}}$, such that the closure of this component intersects $\C(\Gamma)$.
Taking a smaller $\mathcal{N}$ of the same form, we can assume the closure of each component of $\cS_{\overline{m}}\cap\mathcal{N}$ intersects $\C(\Gamma)$, for all the strata 
$\cS_{\overline{m}}$.  For each fixed stratum $\cS_{\overline{m}} \neq \C(\Gamma)$ and two  decorated curves $(C_t, l_t)$ and  $(C_{t'}, l_{t'})$ that lie in the same connected component $\cS_{\overline{m}}\cap \mathcal{N}$,  the Milnor fillings 
that correspond to $(C_t, l_t)$ and  $(C_{t'}, l_{t'})$
are diffeomorphic by a local triviality argument as above.
Now, we observe that for each $S_{\overline{m}}$ there are only finitely many
connected components of  $S_{\overline{m}} \cap \mathcal{N}$. 
Indeed, since every ${\mathcal S}_ {\overline{m}}$
is quasiprojective, it has finitely many irreducible (or connected)
components. Assume that
  ${\mathcal S}_ {\overline{m}}\not= {\mathcal C}(\Gamma)$. 
  Let  ${\mathcal S}'_ {\overline{m}}$ be one of the irreducible
  components of  ${\mathcal S}_ {\overline{m}}$.
  Write
  $X= \overline {{\mathcal S}'_ {\overline{m}}} \cap {\mathcal C}(\Gamma) $.
  It is a quasiprojective space, let $X=\cup_i X_i$ be its irreducible decomposition
  (which is finite).
  Set $ \dim X_i= d_i$ and intersect $X_i$  with a generic $d_i$-codimensional linear germ (cut) $L_i$ and write
  $L_i\cap X_i=p_i$. Then the union of germs $ (\overline {{\mathcal S}'_ {\overline{m}}}\cap L_i,p_i) $ intersects each irreducible component of
  $ {\mathcal S}_ {\overline{m}}\cap {\mathcal N} $ (for ${\mathcal N}$ small). 
  But each germ $ (\overline {{\mathcal S}'_ {\overline{m}}}\cap L_i,p_i) $
  has finitely many analytic local irreducible components, because the local analytic ring
  ${\mathcal O}_{{\mathcal C},p_i} $ is Noetherian.
 Now, since there are only finitely many possible $\overline{m}$ types and therefore finitely many strata 
 $\{S_{\overline{m}}\}_{\overline{m}}$, we conclude that the number of possible Milnor fibers (up to diffeomorphism) is finite.

\bigskip
\bibliography{references1}
\bibliographystyle{plain}

\end{document}